\documentclass[leqno]{article}
\usepackage{CJK,CJKnumb,CJKulem,times,dsfont,ifthen,mathrsfs,latexsym,amsfonts,color}
\usepackage{amsmath,amsthm,makeidx,fontenc,amssymb,bm,graphicx,psfrag,listings,curves,extarrows,enumitem}
\usepackage{hyperref}

\usepackage{geometry}
\usepackage{indentfirst}
\usepackage{amssymb}
\makeatletter

\newcommand{\Rmnum}[1]{\expandafter\@slowromancap\romannumeral #1@}
\makeatother

\usepackage[showonlyrefs]{mathtools}  
\mathtoolsset{showonlyrefs=true}

\newtheorem{theorem}{Theorem}[section]
\newtheorem{lemma}{Lemma}[section]

\newtheorem{proposition}{Proposition}[section]
\newtheorem{remark}{Remark}[section]

			\newcommand{\N}{{\mathbb N}}

			\newcommand{\R}{{\mathbb R}}
			\newcommand{\bean}{\begin{eqnarray*}}
				\newcommand{\eean}{\end{eqnarray*}}

			\newcommand{\sbr}[1]{\left(#1\right)}
			\newcommand{\mbr}[1]{\left[#1\right]}
			\newcommand{\lbr}[1]{\left\{#1\right\}}
			
			\newcommand{\abs}[1]{\left\lvert#1\right\rvert}
			
			\newcommand{\Ni}[1]{\|#1\|_i^2}

			\setitemize{itemindent=38pt,leftmargin=0pt,itemsep=-0.4ex,listparindent=26pt,partopsep=0pt,parsep=0.5ex,topsep=-0.25ex}

\numberwithin{equation}{section}
			
\begin{document}
\theoremstyle{plain}

\title{\bf Least energy positive soultions for $d$-coupled  Schr\"{o}dinger systems with critical exponent in dimension three\thanks{Supported NSFC(No.12171265). E-mail addresses: liuth19@mails.tsinghua.edu.cn (T. H. Liu),    yous16@lzu.edu.cn (S.You), zou-wm@mail.tsinghua.edu.cn (W. M. Zou)} }
				
\date{}
\author{
{\bf Tianhao Liu, \; Song You, \; Wenming Zou}\\
\footnotesize \it  Department of Mathematical Sciences, Tsinghua University, Beijing 100084, China.\\
				}

\maketitle
				
\begin{center}
\begin{minipage}{120mm}
\begin{center}{\bf Abstract }\end{center}		
In the present paper, we consider the coupled Schr\"{o}dinger systems with critical exponent:
\begin{equation*}
\begin{cases}
-\Delta u_i+\lambda_{i}u_i=\sum\limits_{j=1}^{d} \beta_{ij}|u_j|^{3}|u_i|u_i  \quad ~\text{ in } \Omega,\\
u_i \in H_0^1(\Omega) ,\quad i= 1,2,...,d.
\end{cases}
\end{equation*}
Here, $\Omega\subset \mathbb{R}^{3}$ is a smooth bounded domain, $d \geq 2$, $\beta_{ii}>0$ for every $i$, and $\beta_{ij}=\beta_{ji}$ for $i \neq j$. We study a Br\'{e}zis-Nirenberg type problem: $-\lambda_{1}(\Omega)<\lambda_{1},\cdots,\lambda_{d}<-\lambda^*(\Omega)$, where $\lambda_{1}(\Omega)$ is the first eigenvalue of $-\Delta$ with Dirichlet boundary conditions and $\lambda^*(\Omega)\in (0, \lambda_1(\Omega))$. We acquire the existence of least energy positive solutions to this system for weakly cooperative case ($\beta_{ij}>0$ small) and for purely competitive case ($\beta_{ij}\leq 0$) by variational arguments. The proof is performed by mathematical induction on the number of equations, and requires more refined energy estimates for this system. Besides, we present a new nonexistence result, revealing some different phenomena comparing with the higher-dimensional case $N\geq 5$. It seems that this is the first paper to give a rather complete picture for the existence of least energy positive solutions to critical Schr\"{o}dinger system in dimension three.

	\vskip0.23in
	
	
	\vskip0.1in
	{\bf 2010 Mathematics Subject Classification:}
	\vskip0.23in					
\end{minipage}
\end{center}

\vskip0.23in
\section{Introduction}
Consider the following elliptic system with $d \geq 2$ equations
\begin{equation} \label{mainequ1}
\begin{cases}
-\Delta u_i+\lambda_{i}u_i=\sum\limits_{j=1}^{d} \beta_{ij}|u_j|^{p}|u_i|^{p-2}u_i  \quad ~\text{ in } \Omega,\\
u_i \in H_0^1(\Omega) ,\quad i= 1,2,...,d,
\end{cases}
\end{equation}
where  $ N\geq 3$,  $2p \in (2,2^*] $, $2^*=\frac{2N}{N-2}$ is the Sobolev critical exponent, $\beta_{ii}>0$ for every $i$, $\beta_{ij}=\beta_{ji}$ when $i\neq j$.				
System \eqref{mainequ1} appears when looking for standing wave solutions $\Psi_i(x,t)=e^{\imath \lambda_{i}t}u_i(x)$ of time-dependent coupled nonlinear Schr\"odinger system
$$
\imath \partial_t \Psi_i +\Delta \Psi_i +  \sum_{j=1}^{d}\beta_{ij}\abs{\Psi_j}^{p}\abs{\Psi_i}^{p-2}\Psi_i=0, ~i=1,\ldots,d,
$$
where $\imath$ is the imaginary unit. This system originates from many physical models; for example, system \eqref{mainequ1} can be used to explain Bose-Einstein condensation (see \cite{Timmermans 1998}). In quantum mechanics, the solutions $\Psi_i(i=1,...d)$ are the corresponding condensate amplitudes, $\beta_{ii}$ represent self-interactions within the same component, while $\beta_{ij}$ $(i\neq j)$ describe the strength and  type of interactions between different components $u_i$ and $u_j$. Futhermore, $\beta_{ij}>0$ means the interaction is cooperative, while $\beta_{ij}<0$ represents the interaction is competitive.

\vskip0.3in

Set $\mathbb{H}_d:=\sbr{H^1_0(\Omega)}^d$. Note that $\beta_{ij}=\beta_{ji}$, then
solutions of \eqref{mainequ1} correspond to the critical points of the  $C^2$- energy functional $J:\mathbb{H}_d\to \R$ defined by				
\begin{align*}
J(\mathbf{u} )=\frac{1}{2}\sum_{i=1}^{d} \left\|u_i \right\|_i^2-\frac{1}{6}\sum_{i,j=1}^{d} \int_{\Omega}\beta_{ij} \abs{u_i}^p\abs{u_j}^p,
\end{align*}
where $\mathbf{u}=\sbr{u_1,\cdots,u_d}$ and $\|u_i\|_i^2:=\int_{\Omega}(|\nabla u_{i}|^{2}+\lambda_{i}u_{i}^{2})$.

\vskip0.23in
We say a solution  is  {\it trivial } if  all its components  are vanishing. We say a solution is {\it semi-trivial } if there exist at least one (but not all) vanishing component. We say a solution  is  {\it nontrivial} if all its components are nontrivial. However, we are interested in the existence of {\it positive solutions},				i.e., $\mathbf{u}$ solving \eqref{mainequ1} such that $u_i>0$ for every $i$. In particular, we mainly focus on the existence of   {\it  positive least energy  solutions  (or positive ground state), } which attain
\begin{equation}\label{leastenergylevel}
\mathcal{C}_{LES}:=\inf\left\lbrace J(\mathbf{u}): \mathbf{u} \text{ is a solution of } \eqref{mainequ1} \text{ such that }  u_i >0 \text{ for all } i=1,2,...,d \right\rbrace.
\end{equation}
Since the system may admit many semi-trivial solutions,  we will also consider
\begin{equation}\label{eqgroundstatelevel}
\inf\{J(\mathbf{u}):\  J'(\mathbf{u})= 0,\ \mathbf{u}\in H^1_0(\Omega;\R^d), \mathbf{u}\neq \mathbf{0} \}.
\end{equation}
We call a solution $\mathbf{u}\neq \mathbf{0}$ is a {\it generalized ground state solution} if it achieves \eqref{eqgroundstatelevel}.

\medbreak

\vskip0.23in

In the last twenty years, for the subcritical case $2<2p<2^*$, the existence of solutions to \eqref{mainequ1} has been investigated extensively. For the two equations case $d=2$, where there is only one interaction constant $\beta=\beta_{12}=\beta_{21}$, see \cite{Ambrosetti 2007,Bartsch-Dancer-Wang 2010,Chen-Zou2013,Mandel=NoDEA=2015,Sirakov 2007,WT 2008} and reference therein. For an arbitrary number of equations $d\geq 3$, starting from Lin and Wei \cite{Lin-Wei=CMP=2005}, where the authors presented the nonexistence of least energy positive solutions for the purely competitive case and the existence of least energy positive solutions for the purely cooperative case with some additional conditions. In \cite{Correia/Oliveria/Tavares=JFA=2016} the authors studied the existence and nonexistence of positive ground state solutions to system \eqref{mainequ1} for the purely cooperative case. For the mixed case, that is, the existence of at least two pairs, $(i_1,j_1)$ and $(i_2,j_2)$, such that $i_1\neq j_1, i_2\neq j_2$, $\beta_{i_1 j_1}>0$ and $\beta_{i_2 j_2}<0$, the existence of solutions has attracted great interest, see \cite{BSWang 2016,BLWang 2019,Correia/Oliveria/Tavares=JFA=2016,Sato-Wang 2015,soave-hugo=JDE=2016,Soave 2015}.

\vskip0.3in

Different from the subcritical equation, we are more concerned about the critical equation in this article, i.e., $2p=2^*$. For the single equation case $d=1$, the system \eqref{mainequ1} turns into the classical Br\'ezis-Nirenberg problem \cite{Brezis-Nirenberg1983}, where the existence of a positive ground state solution is shown for $-\lambda_1(\Omega)<\lambda_i<0$ when $N\geq 4$. However, in sharp contrast to the high-dimensional situation, from the pioneering paper \cite{Brezis-Nirenberg1983} we learn that there exist  essential differences and difficulties in the three-dimensional case ($N=3$).

 \vskip0.1in

 When $N=3$, system \eqref{mainequ1} reduces to the following problem
\begin{equation}\label{B-N}
	-\Delta u+\lambda_i u= \beta_{ii} |u|^4u, ~u\in H_0^1(\Omega).
\end{equation}
In \cite{Brezis-Nirenberg1983}, the authors proved that \eqref{B-N} has a least energy positive solution $\omega_i\in C^2(\Omega)\cap C^1(\overline{\Omega})$ if $\lambda_i\in (-\lambda_1(\Omega), -\lambda^*(\Omega))$, where $\lambda^*(\Omega)=\frac{\pi^2}{4R_0^2}$ with $R_0=\sup\left\{R| x\in \Omega, B_R(x)\subset \Omega\right\}$, and moreover
\begin{equation}\label{Energy-BN}
	m_i:=\frac{1}{2}\int_\Omega(|\nabla \omega_i|^2+\lambda_i \omega_i^2)-\frac{1}{6}\int_\Omega \mu_i |\omega_i|^6=\frac{1}{3}\|\omega_i\|_i^2 <
	\frac{1}{3} \beta_{ii}^{-\frac{1}{2}}\widetilde{S}^{\frac{3}{2}},
\end{equation}
where $\widetilde{S}$ is the Sobolev best constant of $\mathcal{D}^{1,2}(\R^3)\hookrightarrow L^6(\R^3)$.
 In particular, when $\Omega$ is a ball $B$ in $\R^3$, the authors \cite{Brezis-Nirenberg1983} presented that \eqref{B-N} admits a least energy positive solution if and only if $\lambda_i\in (-\lambda_1(B), -\frac{\lambda_1(B)}{4})$. In \cite{Chen-Zou2012}, the authors proved that \eqref{B-N} has a ground state solution with
 $-\lambda_{n+1}(\Omega)< \lambda_i<-\lambda_{n+1}(\Omega)+\widetilde{S}|\Omega|^{-\frac{2}{N}}$, where $n\geq1$ and $\lambda_{n+1}(\Omega)$ is the $n+1$-th Dirichlet eigenvalue of $(-\Delta, \Omega)$ with multiplicity. For more results related to the Br\'ezis-Nirenberg problem, see \cite{Atkinson-Brezis-Peletier,Cerami-Solimini-Struwe 1986,Clapp-Weth 2005,Roselli-Willem 2009,Schechter-Zou}.

\vskip0.3in

When $N=4 $ and  the system \eqref{mainequ1} consists of two equations (that is, $d=2$),   Chen and Zou \cite{Zou 2012} considered
the problem under the assumption that there is only one interaction constant $\beta=\beta_{12}=\beta_{21}$. Then they proved  that there exist $0<\beta_{1}<\beta_{2}$ such that system \eqref{mainequ1} has a least energy positive solution if $\beta\in (-\infty, \beta_{1}) \cup (\beta_{2}, +\infty)$ when $N=4$. Subsequently, Chen and Zou \cite{Zou 2015} showed that system \eqref{mainequ1} has a least energy positive solution for any $\beta \neq 0$ when $N\geq 5$.

\vskip0.12in

When the number of the system \eqref{mainequ1} is  $d\geq 3$ and the dimension $N\leq 4$,   Guo, Luo and Zou \cite{GuoZou2018}   studied  the pure cooperative system defined   on a bounded smooth domain of $\R^N$, they obtained the existence and classification of  the least energy positive solutions to \eqref{mainequ1} under  the hypotheses $-\lambda_1(\Omega)<\lambda_{1}=\cdots=\lambda_{d}<0$   and  some additional technical conditions on the coupling coefficients. We remark that when $N=2,3$ the system considered in \cite{GuoZou2018}  is subcritical.

\vskip0.12in
When $N\geq 4$ and $2p=2^*$, in \cite{ClappSzulkin,Wu Y-Z 2017} the authors obtained  that existence of least energy positive solutions in a bounded smooth domain of $\R^N$ for the purely competitive cases. While for the purely cooperative case and $N\geq 5$,  Yin and Zou \cite{Yin-Zou} obtained the existence of positive ground state solutions to \eqref{mainequ1}. Recently, Tavares and You \cite{TavaresYou2020} dealt with the existence of least energy positive solutions for the mixed case in a bounded smooth domain of $\R^4$. Afterwards, Tavares, You and Zou \cite{Hugo-You-Zou=Arxiv} established the existence of least energy positive solutions for the mixed case with $N\geq 5$. For the other topics regarding critical system, see \cite{Clapp-Pistoia2018,HeYang2018,Pistoia 2018-1,Tavares 2019}.

\vskip0.12in

All the papers that deal with system \eqref{mainequ1} with $d\geq 2$ in the critical case ($ 2p=2^* $) mainly focus on the higher dimensional case $(N\geq 4)$.
To the best of our knowledge, there are only three papers \cite{Kim2013,YePeng2014,You-Zou2022} studying system \eqref{mainequ1} for the {\bf critical case} on a smooth bounded domain with $N=3$ and $d=2$ in the literature. In \cite{Kim2013,YePeng2014} the authors proved that there exists $\overline{\beta}>0$ such that system \eqref{mainequ1} has a
least energy positive solution if $\beta_{12}>\overline{\beta}$. Recently, You and Zou \cite{You-Zou2022}
proved that system \eqref{mainequ1} has a least energy positive solution for $\beta_{12}>0$ small. Those papers mentioned above only deal with the purely cooperative case ($\beta_{12}>0$).

 \newpage

 As far as we know, there is no paper  considering the existence of  least energy positive solutions of \eqref{mainequ1} with $N=3$ and $2p=2^*$ for the purely competitive cases ($\beta_{12}<0$) or multi equation coupling  case ($d\geq 3$). The present paper makes a first contribution in this directions. We consider the following critical system
\begin{equation} \label{mainequ}
\begin{cases}
-\Delta u_i+\lambda_{i}u_i=\sum\limits_{j=1}^{d} \beta_{ij}|u_j|^{3}|u_i|u_i  \quad ~\text{ in } \Omega \subset \R^3,\\
u_i \in H_0^1(\Omega) ,\quad i= 1,2,...,d.
\end{cases}
\end{equation}				
Throughout this text we always work under the following assumptions				
\begin{equation}\label{assumption-1}
-\lambda_{1}(\Omega)< \lambda_{1},...,\lambda_d <-\lambda^*(\Omega),  \quad \Omega \text{ is a bounded smooth domain of  } \R^3,
\end{equation}
and
\begin{equation}\label{assumption -2}
\beta_{ii} > 0 \quad \forall i=1,2,...,d, \quad \beta_{ij}=\beta_{ji} \quad \forall i,j = 1,2,...,d, i \neq j,
\end{equation}
where  $\lambda_{1}(\Omega)$ denotes the first eigenvalue of $-\Delta$ with Dirichlet boundary conditions. We note  $\lambda^*(\Omega)\in (0, \lambda_1(\Omega))$.

				
\subsection{Main results}

Consider the Nehari type set
\begin{equation} \label{nehari}
\mathcal{N}=\lbr{\mathbf{u}\in \mathbb{H}_d:u_i \not\equiv 0,\left\|u_i \right\|_i^2=\sum_{j=1}^{d} \int_{\Omega}\beta_{ij} \abs{u_i}^3\abs{u_j}^3  \quad \text{for every } i=1,2,..,d},
\end{equation}
and the infimum of $J$ on the set $\mathcal{N}$
\begin{equation} \label{defi of C}
	\mathcal{C}:= \inf_{\mathbf{u}\in \mathcal{N}}J(\mathbf{u})=\inf_{\mathbf{u}\in \mathcal{N}}\frac{1}{3}\sum_{i=1}^d \left\|u_i \right\|_i^2.
\end{equation}
It is easy to see that $\mathcal{C}_{LES}=\mathcal{C}$ if $\mathcal{C}$ is attained on $\mathcal{N}$, where $\mathcal{C}_{LES}$ is defined in \eqref{leastenergylevel}.
\medbreak

Our first result of this paper is the following
\begin{theorem}\label{existence-1}
Assume that \eqref{assumption-1} and  \eqref{assumption -2} hold. There exists a constant $K=K\big(\Omega, \{\lambda_{i}\}_{i=1}^d, \{\beta_{ii}\}_{i=1}^{d}\big) > 0$ such that if
$$
 0<\beta_{ij} < K \quad  \forall i,j = 1,2,...,d, i \neq j,
$$
then $\mathcal{C}$ is achieved by a positive $\mathbf{u}\in \mathcal{N}$, and the system \eqref{mainequ} has a  positive least energy  solution.
\end{theorem}

\begin{remark}
{\rm  Theorem \ref{existence-1} shows that the system \eqref{mainequ} has a least energy positive solution when the interactions between different components are weakly cooperative. We mention that $K$ is only dependent on $\Omega,\lambda_{i},\beta_{ii}, i=1,\ldots, d$. In particular, we will see  \eqref{defi of K}  ahead  for the explicit expression of $K$.}
\end{remark}

\begin{remark}
{\rm When $N=4$ and $p=2$, in \cite{TavaresYou2020} the authors have proved that the system \eqref{mainequ1} has a least energy positive solution for the weakly cooperative case, see \cite[Corollary 1.6]{TavaresYou2020}.}

\end{remark}

In \cite{TavaresYou2020}, because of lack of compactness, the authors established some precise energy estimates and compared the least energy level to \eqref{mainequ1} with that of some kinds of limit system ($\Omega=\mathbb{R}^N$ and $\lambda_i=0$) and appropriate subsystem. In order to obtain the corresponding energy estimate, the authors in \cite{TavaresYou2020} took a cutoff function $\xi$ such that $\xi U \in H_0^1(\Omega)$, where $U$ is the Aubin-Talenti bubble (see \eqref{defi of U{varepsilon}}). Unlike the higher dimensional case $N\geq 4$, the Aubin-Talenti bubble $U$ decays slowly in dimension three. Therefore, $\xi$ can only be chosen as some particular functions for $N=3$ (this fact has been implicitly pointed out by \cite{Brezis-Nirenberg1983}). So it is difficult to acquire the corresponding energy estimates for the system \eqref{mainequ} with $d\geq3$ by using the method in \cite{TavaresYou2020}. Thus we need to introduce new ideas to deal with this problem. In this paper, we compare the least energy level to \eqref{mainequ} with that of single equation ($d=1$) and appropriate subsystem, and establish new energy estimates (see Proposition \ref{Energy-comparing 1} and Proposition \ref{energy comparing c_q}).

\medbreak

To study the existence of ground state solutions of \eqref{mainequ}, we consider the following  Nehari manifold
\begin{equation}
	\mathcal{M}:=\lbr{\mathbf{u}\in \mathbb{H}_d:\  \mathbf{u}\neq \mathbf{0},\  \sum_{i=1}^d\Ni{u_i}=\sum_{i,j=1}^d\int_{\Omega}\beta_{ij}|u_i|^3|u_j|^3 },
\end{equation}
and the level
\begin{equation}
	\mathcal{A}:=\inf\lbr{J(\mathbf{u}): \mathbf{u} \in \mathcal{M}}.
\end{equation}	
Observe that a solution $\mathbf{u}\neq \mathbf{0}$ is a {\it generalized ground state solution} if it achieves $\mathcal{A}$. It follows from  $\mathcal{N}\subset \mathcal{M}$ that  $\mathcal{A} \leq \mathcal{C}$.

\vskip0.12in
Theorem \ref{existence-1} shows that the system \eqref{mainequ} has a positive least energy  solution when the interactions between different components are weakly cooperative. While the next theorem shows that the system \eqref{mainequ} does not have any nontrivial generalized  ground state solution  for the weakly cooperative case.		
\begin{theorem}\label{existence-4}
Assume that \eqref{assumption-1}, \eqref{assumption -2} hold, and $\beta_{ij}\geq 0$ for any $i\neq j$. Then $\mathcal{A}$ is attained and system \eqref{mainequ} has a generalized ground state solution. However,
if $\beta_{ij}\equiv b$, for any $i\neq j$, and
\begin{equation*}
0<b < 2^{\frac{3-d}{2}} \sqrt{\max_{1\leq i \leq d}\lbr{\beta_{ii}}\min_{1\leq i \leq d}\lbr{\beta_{ii}}},
\end{equation*}
then system $(\ref{mainequ})$ has no nontrivial generalized ground state solutions, i.e., strictly we have ${\mathcal{A}}<{\mathcal{C}}$ and
$\mathcal{A}$ is attained only by a semi-trivial element.

\end{theorem}

\begin{remark}
	{\rm In \cite{Hugo-You-Zou=Arxiv,Yin-Zou} the authors proved that system \eqref{mainequ1} with $2p=2^*$ and $N\geq 5$ has a positive generalized ground state solution for any $\beta_{ij}\geq0$, $i\neq j$ (the purely cooperative case). Hence, the  case of $N=3$ is quite different from the higher-dimensional case $N\geq 5$  about the critical  system \eqref{mainequ1}. }
\end{remark}

\vskip0.3in

In subcritical case, when $d=2$ and $p\geq 2$, the author in \cite{Mandel=NoDEA=2015} showed that system \eqref{mainequ1} does not have any nontrivial generalized  ground state solutions if $\beta_{12}\in (0, b_0)$, where $b_0>0$;
When $d\geq 3$ and $p=2$, the authors in \cite{Correia/Oliveria/Tavares=JFA=2016} proved that system \eqref{mainequ1} does not have any nontrivial generalized  ground state solutions if $\beta_{ij}=b \in (0, b_1)$ (see
\cite[Theorem 1.7]{Correia/Oliveria/Tavares=JFA=2016}), where $b_1>0$.

\vskip0.12in
In subcritical case, it is easy to obtain the existence of ground state solutions. However, lack of compactness makes system \eqref{mainequ} very complicated. In this paper, we acquire the existence of ground state solutions by establishing a new energy estimate (see Lemma \ref{GSS energy estimate 1}). Then inspired from \cite{Correia/Oliveria/Tavares=JFA=2016} we show that the generalized ground state is semi-trivial for the weakly cooperative case. But the authors in \cite{Correia/Oliveria/Tavares=JFA=2016} take full use of the fact $p=2$, and the method can not be used directly to deal with system \eqref{mainequ} (the case $p=3$). Thus, we need some important modifications for our proof.

\medbreak
For the purely competitive case, we have the following theorem.
\begin{theorem}\label{existence-4-1}
Assume that \eqref{assumption-1}, \eqref{assumption -2} hold  and
\begin{equation}
\beta_{ij}\leq 0, ~\text{ for any } i\neq j.
\end{equation}
Then system \eqref{mainequ} has a positive  least energy  solution.
\end{theorem}
		
\begin{remark}
 {\rm Up to our knowledge, Theorem \ref{existence-4-1} is the first result to deal with the existence of  positive least energy  solutions of system \eqref{mainequ} for the purely competitive case.}
\end{remark}

We recall the paper \cite{ClappSzulkin}, where  the authors established energy estimate by induction on the number of equation, then they obtained the existence of least energy positive solutions for the purely competitive case when $N\geq 4$. Here, we established the corresponding energy estimate by using this idea for $N=3$ (see Proposition \ref{Key estimate}). However, we need more precise estimates due to the nature of the three-dimensional Br\'ezis-Nirenberg type problem. To the best of our knowledge, it is first time to establish this energy estimate (see Proposition \ref{Key estimate}) of system \eqref{mainequ} for the purely competitive case.

\subsection{Structure of the paper}

 Section \ref{Sec 2} is devoted to the proof of Theorem \ref{existence-1}, in subsection \ref{Sec 2.1}, we present a uniform energy estimate and some preliminary results; in subsection \ref{Sec 2.2}, we establish new energy estimates, see Proposition \ref{Energy-comparing 1} and Proposition \ref{energy comparing c_q}; in subsection \ref{Sec 2.3}, we will give the proof of Theorem \ref{existence-1} by the method of induction on the number of equations.
Section \ref{Sec 2.4} is devoted to the proof of Theorem \ref{existence-4}, in subsection \ref{Sec 3.1}, we introduce the limit system; in subsection \ref{Sec 3.2}, we give the proof of Theorem \ref{existence-4}.
Section \ref{Sec 4} is devoted to the proof of Theorem \ref{existence-4-1}.

\subsection{Further notations}
$\bullet$ The $L^p(\Omega)$ norms will be denoted by $|\cdot|_p$ , $1\leq p \leq \infty$.

$\bullet$ Set $$(\R^+)^d=\lbr{x=(x_1,...,x_d):x_i>0,\text{ for every } i=1,2,...,d}.$$ For a vector $\mathbf{X}=(x_1,...,x_d) \in \R^d$, denote  the transpose of $\mathbf{X}$ by $\mathbf{X}^T$  and define the norm by
\begin{equation*}
|\mathbf{X}|=\sqrt{x_1^2+\cdots+x_d^2}\ .
\end{equation*}

$\bullet$  For a subset $I\subset \lbr{1,\cdots,d}$ with  $|I|$$=q $, we denote the number of elements in set $I$ by $|I|$  and  define $$(u_i)_{i\in I}=(u_{i1},\cdots,u_{iq}),$$ where $I=\lbr{i_1,\cdots,i_q}$ and  $i_1<i_2<\cdots<i_q$.

$\bullet$ Let  $\widetilde{S}$ be the Sobolev best constant of $\mathcal{D}^{1,2}(\R^3)\hookrightarrow L^6(\R^3)$,
\begin{equation}\label{sobolev constant}
	\widetilde{S}=\inf_{u \in \mathcal{D}^{1,2}(\R^3) \setminus \lbr{0} } \cfrac{\int_{\R^3 } |\nabla u|^2 }{\sbr{\int_{\R^3} |u|^6}^{\frac{1}{3}}},
\end{equation}
where $\mathcal{D}^{1,2}(\R^3)=\lbr{u\in L^2(\R^3): |\nabla u| \in L^2(\R^3)}$ with norm $\left\|u \right\|_{\mathcal{D}^{1,2}}:=\sbr{\int_{\R^3}|\nabla u|^2 }^{\frac{1}{2}} $.

$\bullet$ Let \begin{equation} \label{S}
S:=\inf_{i=1,2,...,d}\inf_{u \in H_0^1(\Omega) \setminus \lbr{0} } \cfrac{\left\| u \right\|_i^2 }{\sbr{\int_{\Omega} |u|^6}^{\frac{1}{3}}}\ .
\end{equation}
Moreover, since $\lambda_{i} \in (-\lambda_{1}(\Omega),-\lambda^*(\Omega))$ we have
\begin{equation*}
S|u|_6^2 \leq \Ni{u} \leq |\nabla u|_2^2  \quad \forall u \in H_0^1(\Omega).
\end{equation*}
				
$\bullet$ We use ``$\to$'' and ``$\rightharpoonup$'' to denote the strong convergence and weak convergence in corresponding space respectively.

$\bullet$  
The capital letter $C$ will appear as a constant which may vary from line to line, and $C_1$, $C_2$, $C_3$ are fixed constants.

\section{ Least energy positive solutions for the weakly cooperative case}\label{Sec 2}

In this section, we present the proof of Theorem \ref{existence-1}.		
Given $I \subseteq \lbr{1,2,...,d}$ with  $|I|=q $, $1 \leq q \leq d$, we consider the following subsystem
\begin{equation}\label{subsystem1}
\begin{cases}
-\Delta u_i+\lambda_{i}u_i=\sum\limits_{j\in I} \beta_{ij}|u_j|^{3}|u_i|u_i  \  ~\text{ in } \Omega, \ i \in I,\\
u_i \in H_0^1(\Omega) ,\  i\in I,
\end{cases}
\end{equation}

and define
\begin{align}
&J_I(\mathbf{u}_I)=\frac{1}{2}\sum_{i\in I} \left\|u_i \right\|_i^2-\frac{1}{6}\sum_{i,j\in I} \int_{\Omega}\beta_{ij} \abs{u_i}^3\abs{u_j}^3 , \\
&\mathcal{N}_I=\lbr{\mathbf{u}_I \in \mathbb{H}_q:u_i \not\equiv 0\ \ \text{ and } \ \  \Ni{u_{i}}-\sum_{j\in I} \int_{\Omega}\beta_{ij}\abs{u_{i}}^3\abs{u_{j}}^3 =0, \ i\in I},\\
& \mathcal{C}_I=\inf_{\mathbf{u}_I\in \mathcal{N}_I}J_I(\mathbf{u_I}) =\inf_{\mathbf{u}_I\in \mathcal{N}_I} \frac{1}{3}\sum_{i\in I} \left\|u_i \right\|_i^2 .
\end{align}
Obviously, we have $\mathcal{C}=\mathcal{C}_{\lbr{1,\cdots,d}}$.

\subsection{Preliminary results} \label{Sec 2.1}

In this subsection, we present some preliminary lemmas, which are used to prove Theorem \ref{existence-1}.
Firstly, we show a uniform energy estimate for level $\mathcal{C}_I$.			
\begin{lemma}\label{energy uniformly estimate}
Take \begin{equation} \label{overline C}
\overline{C}=\frac{d}{3}\max_{1\leq i \leq d} \lbr{\frac{1}{\sqrt{\beta_{ii}}}}\widetilde{S}^{\frac{3}{2}},	\end{equation}
then for every $I \subseteq \lbr{1,2,...,d}$, there holds
$$
\mathcal{C}_I \leq \overline{C}.
$$
\end{lemma}
				
\begin{proof}
We follow the arguments of Lemma 2.1 in \cite{soave-hugo=JDE=2016} and Lemma 3.1 in \cite{TavaresYou2020} to prove this lemma. 
					
For every  $I \subseteq \lbr{1,2,...,d}$ with $|I|=q$,  we take $\widehat{u_i} \not\equiv 0$, $i\in I$ such that $\widehat{u_i}\cdot \widehat{u_j} \equiv 0 $, whenever $i \neq j$. Denote $\widetilde{u_i}=t_i\widehat{u_i}$, where $$t_i= \cfrac{\left\|\widehat{u_i}\right\|_i^{\frac{1}{2}} }{(\beta_{ii})^{\frac{1}{4}}|\widehat{u_i}|_{6}^\frac{3}{2}}  \quad \text{ for every } i\in I,$$ then $\widetilde{u_i} \not\equiv 0 $  and $\sbr{\widetilde{u_i}}_{i\in I} \in \mathcal{N}_I$. Thus, since $\lambda_{i} < -\lambda^*(\Omega)<0$, we infer that
\begin{align*}
\mathcal{C}_I &\leq  J_I(\sbr{\widetilde{u_i}}_{i\in I})=\frac{1}{3}\sum_{i\in I} \Ni{\widetilde{u_i}}=\frac{1}{3}\sum_{i\in I}  t_i^2\Ni{\widehat{u_i}}\\
&<\frac{1}{3} \sum_{i\in I}  \cfrac{1}{\sqrt{\beta_{ii}}} \cfrac{ |\nabla \widehat{u_i}|_2^3}{|\widehat{u_i}|_{6}^3}\\  &\leq \frac{1}{3}\max_{i\in I} \lbr{\frac{1}{\sqrt{\beta_{ii}}}} \sum_{i\in I} \cfrac{ |\nabla \widehat{u_i}|_2^3}{|\widehat{u_i}|_{6}^3}.
\end{align*}
Notice the choice of $ \widehat{u_i}$, we have
                     \begin{align*}
						\mathcal{C}_I &\leq \frac{1}{3}\max_{i\in I} \lbr{\frac{1}{\sqrt{\beta_{ii}}}} \inf_{\substack{\Omega \supset \Omega_i\neq \varnothing ,i\in I \\ \Omega_i \cap \Omega_j = \varnothing , i\neq j}} \sum_{i\in I} \widetilde{S}^{\frac{3}{2}}(\Omega_i).
					\end{align*}
On the other hand, for every open subset $\Omega^\prime$ of $\R^3$, by \cite[Propostion 1.43]{william=1996} we have,
$$
\widetilde{S}(\Omega^\prime)=\inf_{u \in H_0^1(\Omega^\prime) \setminus \lbr{0} } \cfrac{\int_{\Omega^\prime } |\nabla u|^2 }{\sbr{\int_{\Omega^\prime} |u|^6}^{\frac{1}{3}}}=\widetilde{S},
$$
where $\widetilde{S}$ is defined by \eqref{sobolev constant}.
Therefore,			
\begin{align*}
\mathcal{C}_I &\leq \frac{q}{3}\max_{1\leq i \leq d} \lbr{\frac{1}{\sqrt{\beta_{ii}}}} \widetilde{S}^{\frac{3}{2}} \leq \frac{d}{3}\max_{1\leq i \leq d} \lbr{\frac{1}{\sqrt{\beta_{ii}}}}\widetilde{S}^{\frac{3}{2}},
\end{align*}
which yields that $\mathcal{C}_I\leq \overline{C}$, where $\overline{C}$ is defined in \eqref{overline C}.
				\end{proof}
				
				Define
				\begin{equation} \label{K_1}
					K_1=\cfrac{7S^3}{12(6 \overline{C})^2},
				\end{equation}where $S$ is defined in \eqref{S}.
				
\begin{lemma} \label{L^6norm-estimate}
If $$\beta_{ii} > 0 \quad \forall i=1,2,...,d, \quad 0<\beta_{ij} < K_1 \quad  \forall i,j = 1,2,...,d, i \neq j ,$$  and  for every $I \subseteq \lbr{1,2,...,d}$,  $\mathbf{u}\in \mathcal{N}_I$ with $J_I(\mathbf{u})\leq 2 \overline{C}$, then there exists constant $C_2>C_1>0$ dependent only  on $K_1,\lambda_i,\beta_{ii}$, such that $$C_1 \leq \int_{\Omega}|u_i|^6  \leq C_2 \quad \text{ for every } i \in I.$$
				\end{lemma}
				
				\begin{proof}
					For any $\mathbf{u} \in \mathcal{N}_I $ with $J_I(\mathbf{u})\leq 2 \overline{C}$,  we have
					\begin{equation} \label{norm 6C}
						\sum_{i \in I} \Ni{u_i} \leq 6 \overline{C}.
					\end{equation}
Therefore,
$$
S\left(\int_{\Omega}|u_i|^6  \right)^{\frac{1}{3}}\leq \Ni{u_i }\leq \sum\limits_{i\in I}\Ni{u_i} \leq 6 \overline{C},
$$
that is  $ \int_{\Omega}|u_i|^6  \leq C_2$.
On the other hand, we have
\begin{align*}
S\sbr{\int_{\Omega}|u_i|^6 }^{\frac{1}{3}}\leq \Ni{u_i }&=\sum_{j\in I} \int_{\Omega} \beta_{ij}|u_i|^3|u_j|^3 \leq  d\max_{i=1,2,...,d}\lbr{K_1,\beta_{ii}}\left(\frac{6\overline{C}}{S}\right)^{\frac{3}{2}}\sbr{\int_{\Omega}|u_i|^6 }^{\frac{1}{2}},
\end{align*}
which yields that  $ \int_{\Omega}|u_i|^6\geq C_1$.
\end{proof}
	
Before proceeding, we introduce some notations. For every $I \subseteq \lbr{1,2,...,d}$ with  $|I|$$ =q $ , we define  the matrix $A_I(\mathbf{u}) = (a_{ij}(\mathbf{u}))_{(i,j)\in I^2}$ by
				\begin{equation} \label{defi of A_q and B_q}
					\begin{aligned}
						&a_{ii}(\mathbf{u})=4\int_{\Omega}\beta_{ii} |u_i|^6 +\sum_{\substack{j\in I , j\neq i}} \int_{\Omega} \beta_{ij}|u_i|^3|u_j|^3 , \ i \in I, \\
						&a_{ij}(\mathbf{u})=3\int_{\Omega}\beta_{ij}|u_i|^3|u_j|^3 , \ i,j\in I,  i\neq j.\\
					\end{aligned}
				\end{equation}
Set	
\begin{equation*}
\Gamma_I=\lbr{\mathbf{u}\in \mathbb{H}_q:A_I(\mathbf{u}) \text{ is strictly diagonally dominant }}.
\end{equation*}

\begin{remark}\label{remark2.1}
{\rm
(1) 
For any $\mathbf{u} \in \Gamma_I$, we know that $A_I(\mathbf{u})$ is positive definite by Gershgorin circle theorem.
						
(2) Since there holds the embedding $H_0^1(\Omega)\hookrightarrow L^6(\Omega)$, it is not difficult to verify that  $\Gamma_I$ is open in $\mathbb{H}_q$.
						
(3) The set $\mathcal{N}_I \cap \Gamma_I$  is not empty. In fact, following  the proof of Lemma \ref{energy uniformly estimate}, it is easy to see that   $\sbr{\widetilde{u_i}}_{i\in I} \in \mathcal{N}_I \cap \Gamma_I$.
}
\end{remark}

The following lemma shows that $\mathcal{N}_I \cap\Gamma_I$ is a natural constraint  for the weakly cooperative  case.

\begin{lemma}\label{achieve and crtical point}
Assume that
$$
\beta_{ii} > 0 \quad \forall i=1,2,...,d, \quad 0<\beta_{ij} < K_1 \quad  \forall i,j = 1,2,...,d, i \neq j,
$$
then for every  $I \subseteq \lbr{1,2,...,d}$,  the set $\mathcal{N}_I \cap \Gamma_I$ is a smooth manifold.
Moreover, the constrained critical points of $J_I$ on $\mathcal{N}_I\cap\Gamma_I$ are free critical points of $J_I$. In other words, $\mathcal{N}_I \cap\Gamma_I$ is a natural constraint.
				\end{lemma}
				
				\begin{proof}

						For every  $I \subseteq \lbr{1,2,...,d}$ with  $|I|$$ =q $,  we take $\mathbf{u}=\sbr{u_i}_{i\in I}\in \mathcal{N}_I \cap \Gamma_I $, we define
						\begin{equation} \label{defi of G_i}
							G_i(\mathbf{u}):=\Ni{u_{i}}-\sum_{j\in I} \int_{\Omega}\beta_{ij}\abs{u_{i}}^3\abs{u_{j}}^3 , \quad i \in I.
						\end{equation}
					By a direct computation, for every $\mathbf{v}=(v_i)_{i\in I} \in \mathbb{H}_q$ we obtain
					\begin{equation} \label{2.1}
						G'_i(\mathbf{u})\mathbf{v}=2\int_{\Omega}\sbr{ \nabla u_i \cdot \nabla v_i +\lambda_i u_iv_i} -3\sum_{j\in I}\sbr{\int_{\Omega} \beta_{ij} |u_i||u_j|^3u_i v_i+\int_{\Omega} \beta_{ij} |u_j||u_i|^3u_j v_j }.
					\end{equation}
					
					We claim that the set $\mathcal{N}_I \cap \Gamma_I$ is a smooth manifold of codimension $q$ in a neighborhood of $\mathbf{u}$ in $\mathbb{H}_q$. To verify this, we take $\mathbf{u}=\sbr{u_i}_{i\in I}\in \mathcal{N}_I \cap \Gamma_I  $ and  prove  the map $\widehat{G}_{\mathbf{u}}: \mathbb{H}_q \to \R^q$ is a surjective as linear operator, where
					\begin{equation}
						\widehat{G}_{\mathbf{u}}(\mathbf{v})=(G'_i(\mathbf{u})\mathbf{v})_{i\in I}.
					\end{equation}
Note that $\mathbf{u}\in\mathcal{N}_I$. Then take $v_i=-t_iu_i$, we have
					\begin{equation}
			\begin{aligned}
				G'_i(\mathbf{u})\mathbf{v}
								& =\sbr{ -2\Ni{u_i}+3 \sum_{j\in I}\int_{\Omega}\beta_{ij}\abs{u_{i}}^3\abs{u_{j}}^3 }t_i+3\sum_{j\in I}\sbr{\int_{\Omega}\beta_{ij}\abs{u_{i}}^3\abs{u_{j}}^3 } t_j\\
								&=\sbr{ 4\int_{\Omega}\beta_{ii}|u_i|^6+ \sum_{\substack{j\in I , j\neq i}}\int_{\Omega}\beta_{ij}\abs{u_{i}}^3\abs{u_{j}}^3 }t_i+3\sum_{\substack{j \in I , j\neq i}}\sbr{\int_{\Omega}\beta_{ij}\abs{u_{i}}^3\abs{u_{j}}^3  }t_j.
						\end{aligned}
					\end{equation}
					Hence, we see that
\begin{equation}\label{2.6}
(G'_i(\mathbf{u})\mathbf{v})_{i\in I}^T=A_I(\mathbf{u})\mathbf{t},
\end{equation}
where $A_I(\mathbf{u})$ is defined in \eqref{defi of A_q and B_q}, $\mathbf{t}=\sbr{t_i}_{i\in I}^T \in \R^q$ and $\mathbf{v}=(-t_iu_i)_{i\in I}$. Since $\mathbf{u} \in \Gamma_I$,  the matrix $A_I(\mathbf{u})$ is strictly diagonally dominant and from Remark \ref{remark2.1} (1)  that $A_I(\mathbf{u})$ is positive definite. Hence, $A_I(\mathbf{u})$ is non-singular. Then for any $\mathbf{h}=(h_i)_{i\in I} \in \R^q$,  there exists $\mathbf{v}'=(-s_iu_i)_{i\in I}$ such that
					$$	(G'_i(\mathbf{u})\mathbf{v}')_{i\in I}^T=\mathbf{h},$$ where $\mathbf{s}=A_I(\mathbf{u})^{-1}\mathbf{h}$. Therefore, the claim is true.
					
Finally, we show that $\mathcal{N}_I \cap \Gamma_I$ is a natural constraint.
Assume that  $\mathcal{C}_I$ is achieved by $\mathbf{u}=(u_i)_{i\in I}\in \mathcal{N}_I \cap \Gamma_I$. By Remark \ref{remark2.1} (2), the constraint  $\mathcal{N}_I \cap \Gamma_I$ is an open subset of $\mathcal{N}_I$ in the topology of $\mathbb{H}_q$. Thus the function $\mathbf{u}$ is an inner critical point of $J_I$ in an open subset of $\mathcal{N}_I$, and in particular it is a constrained critical point of $J_I$ on $\mathcal{N}_I$.
					Since the set $\mathcal{N}_I \cap \Gamma_I$ is a smooth manifold of codimension $q$ in a neighborhood of $\mathbf{u}$ in $\mathbb{H}_q$, then by the Lagrange multipliers rule there exists $\mu_i\in \R$ , $ i\in I $ such that
					\begin{equation}\label{2.7}
						J_I'(\mathbf{u})-\sum_{i\in I} \mu_iG_i'(\mathbf{u})=0,
					\end{equation}
where $G_i(\mathbf{u})$ is defined in \eqref{defi of G_i}.
Testing  \eqref{2.7} by $(0,...,u_i,...,0)$ for $i\in I$ and thanks to  $G_i(\mathbf{u})=0$ for every $i \in I$, we have
\begin{equation}
\sbr{4\int_{\Omega}\beta_{ii}|u_i|^6  + \sum_{\substack{j\in I ,j\neq i}} \int_{\Omega} \beta_{ij}|u_i|^3|u_j|^3 } \mu_i+ 3\sum_{\substack{j\in I, j\neq i}}\sbr{\int_{\Omega} \beta_{ij}|u_i|^3|u_j|^3 }\mu_j =0.
\end{equation}
					Hence, we have $A_I(\mathbf{u})\boldsymbol{\mu}=\mathbf{0}$, $\boldsymbol{\mu}=(\mu_i)_{i\in I}^T$.  From the above arguments, we know that the matrix $A_I(\mathbf{u})$ is non-singular, then $\mu_i=0$ for all $i\in I$. Combining this with \eqref{2.7}, we have $J_I'(\mathbf{u})=0$. That is, $\mathbf{u}$ is a free critical point of $J_I$ on $\mathbb{H}_q$, which means that $\mathcal{N}_I \cap \Gamma_I$ is a natural constraint.
				\end{proof}	
					
					\begin{lemma} \label{diagonally dominant}
						Assume that  $$\beta_{ii} > 0 \quad \forall i=1,2,...,d, \quad 0<\beta_{ij} < K_1 \quad  \forall i,j = 1,2,...,d, i \neq j ,$$ then  we have
						$$
						\mathcal{N}_I \cap \lbr{\mathbf{u}\in \mathbb{H}_q: J_I(\mathbf{u})\leq 2 \overline{C}} \subset \Gamma_I.
						$$
						Moreover, the constrained critical points of $J_I$ on $\mathcal{N}_I$ satisfying $J_I(\mathbf{u})\leq 2 \overline{C}$ are free critical points of $J_I$.
					\end{lemma}
									
					\begin{proof}
						Take $\mathbf{u} \in \mathcal{N}_I \cap \lbr{\mathbf{u}\in \mathbb{H}_q: J_I(\mathbf{u})\leq 2\ \overline{C}}$.	 We will  prove that $A_I(\mathbf{u})$ is strictly diagonally dominant, that is $$4\int_{\Omega}\beta_{ii} |u_i|^6 +\sum_{\substack{j\in I , j\neq i}} \int_{\Omega} \beta_{ij}|u_i|^3|u_j|^3 -3\sum_{\substack{j\in I , j\neq i}}\left| \int_{\Omega} \beta_{ij}|u_i|^3|u_j|^3 \right| > 0 , \ i \in I.$$
						Notice that $\beta_{ij}>0$ and $\mathbf{u}\in \mathcal{N}_I$ , we only need to show
						$$4\Ni{u_i}-6\sum_{\substack{j\in I , j\neq i}} \int_{\Omega} \beta_{ij}|u_i|^3|u_j|^3 >0, \ i \in I.$$	
						In fact, thanks to the choice of $K_1$, we have
						\begin{align*}
							6\sum_{\substack{j\in I , j\neq i}} \int_{\Omega} \beta_{ij}|u_i|^3|u_j|^3
							\leq \frac{6K_1}{S^3}\sum_{\substack{j\in I , j\neq i}} \left\| u_i\right\|_i^3 \left\| u_j\right\|_j^3
							\leq \frac{6K_1}{S^3} (6 \overline{C})^2 \Ni{u_i} \leq \frac{7}{2} \Ni{u_i}.
						\end{align*}
						Thus, by Lemma \ref{L^6norm-estimate} we have
						\begin{equation}
							4\Ni{u_i}-6\sum_{\substack{j\in I , j\neq i}} \int_{\Omega} \beta_{ij}|u_i|^3|u_j|^3  \geq \frac{1}{2} \Ni{u_i} \geq \frac{1}{2} S\sbr{\int_{\Omega} |u_i|^6 }^{\frac{1}{3}}\geq \frac{1}{2}SC_1^{\frac{1}{3}} .
						\end{equation}
						It follows that \begin{equation} \label{es-2}
							4\int_{\Omega}\beta_{ii} |u_i|^6 +\sum_{\substack{j\in I , j\neq i}} \int_{\Omega} \beta_{ij}|u_i|^3|u_j|^3 -3\sum_{\substack{j\in I , j\neq i}} \left|\int_{\Omega} \beta_{ij}|u_i|^3|u_j|^3 \right| \geq \frac{1}{2}SC_1^{\frac{1}{3}} >0,
						\end{equation}
which means that $A_I(\mathbf{u})$ is strictly diagonally dominant. Therefore,
$$
\mathcal{N}_I \cap \lbr{\mathbf{u}\in \mathbb{H}_q: J_I(\mathbf{u})\leq 2 \overline{C}} \subset \Gamma_I,
$$
and so
$$
\mathcal{N}_I \cap \lbr{\mathbf{u}\in \mathbb{H}_q: J_I(\mathbf{u})\leq 2 \overline{C}} \subset \mathcal{N}_I \cap\Gamma_I.
$$
						By Lemma \ref{achieve and crtical point} we know that the constrained critical points of $J_I$ on $\mathcal{N}_I$  satisfying $J_I(\mathbf{u})\leq 2 \overline{C}$ are free critical points of $J$.
						 This completes the proof.
					\end{proof}

Next, we construct a Palais-Smale sequence at level $\mathcal{C}_I$.		
\begin{lemma}{\rm (Existence of Palais-Smale sequence)} \label{exist of ps sequence}
\ Assume that
$$
\beta_{ii} > 0 \quad \forall i=1,2,...,d, \quad 0<\beta_{ij} < K_1 \quad  \forall i,j = 1,2,...,d, i \neq j.
$$
Then for every $I \subseteq \lbr{1,2,...,d}$, there exists a sequence $\lbr{\mathbf{u}_n}\subset \mathcal{N}_I$ satisfying
$$
\lim_{n \to \infty}J_I(\mathbf{u}_n)=\mathcal{C}_I ,\quad \lim_{n \to \infty}J_I^{\prime}(\mathbf{u}_n)=0.
$$
				\end{lemma}
				
				\begin{proof}
						By the definition of $\mathcal{C}_I$, there exists a minimizing sequence  $\lbr{\mathbf{u}_n}\subset \mathcal{N}_I$ with $\mathbf{u}_n=(u_{i,n})_{i\in I}$ satisfying
					\begin{equation} \label{minimizing sequence}
						J_I(\mathbf{u}_n) \to \mathcal{C}_I, \quad J_I^\prime(\mathbf{u}_n)-\sum_{i\in I} \mu_{i,n}G_i^\prime(\mathbf{u}_n)=o(1)
					\end{equation}
					where $$G_i(\mathbf{u})= \Ni{u_{i}}-\sum\limits_{j\in I} \beta_{ij}\abs{u_{i}u_{j}}_3^3.$$
					
					By Lemma \ref{energy uniformly estimate} we can assume that $J_I(\mathbf{u}_n) \leq 2 \overline{C}$ for $n$ large enough, then following lemma \ref{diagonally dominant} we have
					\begin{equation} \label{2.2}
						\begin{aligned}
							&4\beta_{ii} |u_{i,n}|_6^6 +\sum_{\substack{j\in I, j\neq i}} \int_{\Omega} \beta_{ij}|u_{i,n}|^3|u_{j,n}|^3 -3\sum_{\substack{j\in I , j\neq i}}\left| \int_{\Omega} \beta_{ij}|u_{i,n}|^3|u_{j,n}|^3\right| \geq \frac{1}{2}SC_1^{\frac{1}{3}} \ \text{ for  } i\in I.
						\end{aligned}
					\end{equation}
Suppose that $\nu_{n}$ is the minimum eigenvalues of $A_I(\mathbf{u}_n)$. By Gershgorin circle theorem and \eqref{2.2} we have
\begin{equation}\label{2.2-3}
\nu_{n}\geq \frac{1}{2}SC_1^{\frac{1}{3}},
\end{equation}
where $C_1$ is independent on $n$.
						
Note that $\mathbf{u}_n\in \mathcal{N}_I$, then test the second equation in \eqref{minimizing sequence} with $(0,...,u_{i,n},...,0)$, $i\in I$ and multiply by $\boldsymbol{\mu}_{n}=(\mu_{i,n})_{i\in I}$, by \eqref{2.2-3} we have
$$	
o(1)|\boldsymbol{\mu}_{n}|\geq \boldsymbol{\mu}_{n} A_I(\mathbf{u}_n) \boldsymbol{\mu}_{n}^{T}\geq \nu_{n}|\boldsymbol{\mu}_{n}|^2 \geq\frac{1}{2}SC_1^{\frac{1}{3}}  |\boldsymbol{\mu}_{n}|^2,
$$
					where $A_I(\mathbf{u}_n)$ is defined in \eqref{defi of A_q and B_q}.
					It follows that $\mu_{i,n} \to 0$ as $n \to \infty$.
					 Since for every $ \varphi \in H_0^1(\Omega)$, $G_i^{\prime}(\mathbf{u}_n) \varphi$ is uniformly bounded, we have $J_I^{\prime}(\mathbf{u}_n)\varphi =o(\left\| \varphi\right\| )$, which yields that $J_I^{\prime}(\mathbf{u}_n) \to 0 $ in $H^{-1}(\Omega)$. Therefore,  $\lbr{\mathbf{u}_n}$ is a standard Palais-Smale sequence.
				\end{proof}
			We conclude this section by introducing  the Br\'{e}zis-Lieb lemma(see \cite{Brezis Lieb lemma}) for two components, and its proof is referred to \cite[p.447]{Zou 2015}.
		
				\begin{lemma} \label{brezis-lieb lemma}
					Assume that $u_n\rightharpoonup u$, $v_n\rightharpoonup v$ in $H_0^1(\Omega)$ as $n \to \infty$ and $1<p <+\infty$. Then, up to subsequence, there holds
					\[\lim_{n \to \infty} \int_{\Omega} \sbr{|u_n|^p|v_n|^p-|u_n-u|^p|v_n-v|^p-|u|^p|v|^p}=0.\]
				\end{lemma}

				\vskip0.23in
\subsection{Energy estimates} \label{Sec 2.2}
In this subsection, we present two crucial energy estimates, which are important to prove Theorem \ref{existence-1}. The first one is the following proposition, which plays a key role in showing that the limit of Palais-Smale sequence is not zero.	Define
\begin{equation} \label{ defi of K_2}
K_2=\cfrac{\min\limits_{1\leq i \leq d} \left\{\sqrt{\beta_{ii}m_i}\right\}}{2\sum\limits_{i=1}^d \sqrt{\frac{m_i}{\beta_{ii}}}}.
\end{equation}
Then we have

\begin{proposition} \label{Energy-comparing 1}
Assume that there holds   $$\beta_{ii} > 0 \quad \forall i=1,2,...,d, \quad 0<\beta_{ij} < K_2 \quad  \forall i,j = 1,2,...,d, i \neq j, $$ then  we have
\begin{equation} \label{energy of c_d}
\mathcal{C}_I \leq \sum_{i\in I} m_i \quad \forall I \subseteq \{1,\ldots, d\}.
\end{equation}
\end{proposition}
				
\begin{proof}
Without loss of generality, we only prove that
\begin{equation*}
\mathcal{C} \leq \sum_{i=1}^d m_i.
\end{equation*}
 We will prove this statement in three steps.
To begin with, we recall that  $\omega_i$ is a least energy positive solution of the Br\'ezis-Nirenberg problem with energy $m_i=\frac{1}{3}\|\omega_i\|_i^2=\frac{1}{3}\beta_{ii} |\omega_i|_6^6$ (see \eqref{Energy-BN}).
				
\vskip0.2in			
					
\noindent{\bf Step1:} We claim that the matrix $\left(  \int_{\Omega} \beta_{ij}|\omega_i|^3|\omega_j|^3 \right) _{d\times d} $ is positive definite.
					
					 For every $ 1 \leq i \leq d$,
					\begin{equation}
						\begin{aligned}
							\int_{\Omega} \beta_{ii}|\omega_i|^6 - \sum_{\substack{j=1 \\ j\neq i}}^d \left| \int_{\Omega} \beta_{ij}|\omega_i|^3|\omega_j|^3 \right|  &\geq 3m_i-K_2  \sum_{\substack{j=1 \\ j\neq i}}^d \sbr{\int_{\Omega} |\omega_i|^6 }^{\frac{1}{2}}\sbr{\int_{\Omega} |\omega_j|^6 }^{\frac{1}{2}}\\ &\geq 3m_i-K_2 \sqrt{\frac{3m_i}{\beta_{ii}}}\sum_{j=1 }^d \sqrt{\frac{3m_j}{\beta_{jj}}}\geq  \frac{3}{2}m_i>0.
						\end{aligned}
					\end{equation}
This implies that the matrix  $\left(  \int_{\Omega} \beta_{ij}|\omega_i|^3|\omega_j|^3 \right) _{d\times d} $ is strictly diagonally dominant. Since the diagonal elements are positive, then this matrix is positive definite.

\vskip0.2in

\noindent{\bf Step2:} We claim  that there exists $(a_1,...,a_d)\in (\R^+)^d$ such that $(a_1\omega_1,...,a_d\omega_d)\in \mathcal{N}$.

					We define the polynomial function $F:(\R^+)^d \to \R$
					\begin{equation}
						F(t_1,...,t_d)=J(t_1\omega_1,...,t_d\omega_d)=\frac{1}{2}\sum_{i=1}^{d} t_i^2\left\|\omega_i \right\|_i^2-\frac{1}{6}\sum_{i,j=1}^{d} t_i^3t_j^3\int_{\Omega}\beta_{ij} \abs{\omega_i}^3\abs{\omega_j}^3 ,
					\end{equation}
					where $(\R^+)^d=\lbr{x=(x_1,...,x_d):x_i>0 \text{ for } i=1,\cdots,d}$. By using the conclusion of the Step1, there exists a constant $C$ such that
					\begin{equation}
					 	F(t_1,...,t_d) \leq \frac{1}{2}\sum_{i=1}^{d} t_i^2\left\|\omega_i \right\|_i^2-\frac{C}{6}\sum_{i=1}^{d} t_i^6=\frac{3}{2}\sum_{i=1}^{d} m_it_i^2-\frac{C}{6}\sum_{i=1}^{d} t_i^6 \to -\infty \ \  \text{ as } \abs{\mathbf{t}} \to +\infty.
					\end{equation}
					Thus, the polynomial $F(t_1,...,t_d)$ has a global maximum in $\overline{(\R^+)^d}$.
					
					Assume the global maximum points $\mathbf{t}=(a_1,...,a_d)$ belongs to  $\partial \overline{(\R^+)^d}$. Without loss of generality, we assume that $a_1=0$ and $a_i>0$, $\forall i=2,...,d$, then
					$$F(0,a_2,...,a_d)=\frac{1}{2}\sum_{i=2}^{d} a_i^2\left\|\omega_i \right\|_i^2-\frac{1}{6}\sum_{i,j=2}^{d}a_i^3a_j^3\int_{\Omega}\beta_{ij} \abs{\omega_i}^3\abs{\omega_j}^3 .$$
					For  $s>0$ small enough, we have
$$
F(s,a_2,...,a_d)-F(0,a_2,...,a_d)=\frac{1}{2}s^2\left\|\omega_1 \right\|_1^2-\frac{1}{6}s^6\int_{\Omega}\beta_{11}|\omega_1|^6-\frac{1}{3}s^3\sum_{j=2}^{d} a_j^3\int_{\Omega}\beta_{1j} \abs{\omega_1}^3\abs{\omega_j}^3>0,
$$
					which contradicts to the fact that $(0,a_2,...,a_d)$ is a global maximum of $\overline{(\R^+)^d}$. Thus,  the global maximum point of $ 	F(t_1,...,t_d) $ can not belong to  $\partial \overline{(\R^+)^d}$, which implies that the global maximum point of $F(t_1,...,t_d)$ is a interior point  in $ (\R^+)^d $. Moreover, the global maximum point $ (a_1,...,a_d) \in (\R^+)^d  $ is a critical point, which means that
					\begin{equation}
					\frac{\partial F}{\partial t_i}(a_1,...,a_d)=0, \text{ for every } i=1,2,...,d.
					\end{equation}
					
Therefore,
				\begin{equation*}
				(a_1\omega_1,...,a_d\omega_d)\in \mathcal{N}.
\end{equation*}

\vskip0.2in
			
\noindent{\bf Step3:} We claim that $\mathcal{C}\leq \sum\limits_{i=1}^d m_i$.
			By the definition of $\omega_i$ and $\beta_{ij}>$ for any $i\neq j$ we see that
			\begin{equation}
				\begin{aligned}
					\mathcal{C} & \leq J(a_1\omega_1,...,a_d\omega_d)=\frac{1}{2}\sum_{i=1}^{d} a_i^2\left\|\omega_i \right\|_i^2-\frac{1}{6}\sum_{i,j=1}^{d} a_i^3a_j^3\int_{\Omega}\beta_{ij} \abs{\omega_i}^3\abs{\omega_j}^3 \\ &\leq \frac{1}{2}\sum_{i=1}^{d} a_i^2\left\|\omega_i \right\|_i^2-\frac{1}{6}\sum_{i=1}^{d} a_i^6\int_{\Omega}\beta_{ii} \abs{\omega_i}^6\\ &=\sum_{i=1}^{d}\left( \frac{3}{2}a_i^2-\frac{1}{2}a_i^6\right) m_i \leq \sum\limits_{i=1}^d m_i.
				\end{aligned}
			\end{equation}
		This completes the proof.
			\end{proof}

\medbreak
				
The following proposition will play a critical role in proving that $\mathcal{C}$ is achieved by a solution with $d$ nontrivial components. Define
				
\begin{equation} \label{defi of K_3}
					K_3= \min \lbr{K_1,\  \cfrac{S^3}{4(6 \overline{C})^2}, \  \cfrac{S^\frac{5}{2}C_1^{\frac{1}{3}}}{4\sbr{6\overline{C}}^\frac{3}{2}\sum\limits_{i=1}^d\sqrt{\frac{3m_i}{\beta_{ii}}}}\ ,\  \cfrac{\frac{\sqrt{3}}{2}\min\limits_{1\leq i \leq d}\sqrt{\beta_{ii}m_i}}{\sum\limits_{i=1}^d\sqrt{\frac{3m_i}{\beta_{ii}}}+\sbr{\frac{6\overline{C}}{S}}^{\frac{3}{2}}}},
				\end{equation}
				where  $K_1$ is defined in \eqref{K_1}. Then we have the following energy estimate.
				\begin{proposition}\label{energy comparing c_q}
					Assume that there holds $$\beta_{ii} > 0 \quad \forall i=1,2,...,d, \quad 0<\beta_{ij} < K_3 \quad  \forall i,j = 1,2,...,d, i \neq j.$$
				
Given $I \subseteq \lbr{1,2,...,d}$, suppose that $ \mathcal{C}_Q \text{ is achieved by } \mathbf{u}_Q$ for every $Q\subsetneq I$, then
					\begin{equation}
						\mathcal{C}_I \leq \min \lbr{\mathcal{C}_Q+\sum_{\substack{i\in I\backslash Q}} m_i:Q\subsetneq I }.
					\end{equation}
				\end{proposition}
Next, we present the proof of this proposition. Without loss of generality, we fix $1 \leq  q \leq  d-1 $ and prove that
			\begin{equation} \label{es-1}
					\mathcal{C} \leq \mathcal{C}_{1,\ldots,q}+\sum_{i=q+1}^d m_i,
			\end{equation}
		where we use the notation $J_{1\cdots,q}$, $ \mathcal{N}_{1,\ldots,q} $, $\mathcal{C}_{1,\ldots,q}$ instead of $J_{\lbr{1\cdots,q}}$, $ \mathcal{N}_{\lbr{1,\ldots,q}} $,  $\mathcal{C}_{\lbr{1,\cdots,q}}$ for simplicity, and the other inequalities can be proved in the same way.	Before proving \eqref{es-1}, let us firstly prove the following Lemma \ref{subsystem-maximum} and Lemma \ref{esist of t_1,t_d}.
							
				\begin{lemma} \label{subsystem-maximum}
					Assume that there holds $$\beta_{ii} > 0 \quad \forall i=1,2,...,d, \quad 0<\beta_{ij} < K_3 \quad  \forall i,j = 1,2,...,d, i \neq j.$$
					Given $1\leq q \leq d-1$, if $\mathcal{C}_{1,\ldots,q}$ is achieved by $\mathbf{u}_q=(u_1,...,u_q) \in \mathcal{N}_{1,\ldots,q}$, then $$\max_{t_1,...,t_q>0}f_q(t_1,...,t_q)=f_q(1,...,1)=\mathcal{C}_{1,\ldots,q}.$$
				\end{lemma}

				\begin{proof}

Notice that  $\mathcal{C}_{1,\ldots,q}$ is achieved by $\mathbf{u}_q=(u_1,...,u_q) \in \mathcal{N}_{1,\ldots,q}$, then by Lemma \ref{energy uniformly estimate} we have $J_{1,\cdots,q}(\mathbf{u_q}) =\mathcal{C}_{1,\ldots,q} < 2\overline{C}$. Consider the polynomial function $f_q:(\R^+)^q \to \R   $
				\begin{equation} \label{defi of f}
					f_q(t_1,...,t_q)=J_{1,\cdots,q}(t_1u_1,...,t_qu_q):=\frac{1}{2}\sum_{i=1}^{q} t_i^2\left\|u_i \right\|_i^2-\frac{1}{6}\sum_{i,j=1}^{q} t_i^3t_j^3\int_{\Omega}\beta_{ij} \abs{u_i}^3\abs{u_j}^3.
				\end{equation}
Define the matrix $B_q(\mathbf{u}_q)=(b_{ij}(\mathbf{u}_q))$ by
\begin{equation}\label{Bqmatrix}
b_{ij}(\mathbf{u}_q) =\int_{\Omega} \beta_{ij}|u_i|^3|u_j|^3 ,\quad  i,j=1,2,...,q.
\end{equation}
We claim that the matrix $B_q(\mathbf{u}_q)$ is positive definite. We will prove that $B_q(\mathbf{u})$ is strictly diagonally dominant, that is \begin{equation} \label{Bq dominant}
		\int_{\Omega}\beta_{ii} |u_i|^6  -\sum_{\substack{j=1 \\ j\neq i}}^q\left|\int_{\Omega}  \beta_{ij}|u_i|^3|u_j|^3\right|  > 0 .
	\end{equation}
	Note that $\mathbf{u}_q\in \mathcal{N}_{1,\ldots,q}$, then the inequality \eqref{Bq dominant} is true if we show $$\Ni{u_i}-2\sum_{\substack{j=1 \\ j\neq i}}^q \int_{\Omega} \beta_{ij}|u_i|^3|u_j|^3 >0.$$
By the definition of $K_3$ we have
	\begin{align*}
		2\sum_{\substack{j=1 \\ j\neq i}}^q\int_{\Omega} \beta_{ij}|u_i|^3|u_j|^3& \leq \frac{2K_3}{S^3}\sum_{\substack{j=1 \\ j\neq i}}^q \left\| u_i\right\|_i^3 \left\| u_j\right\|_j^3\\
		&\leq \frac{2K_3}{S^3} (6 \overline{C})^2 \Ni{u_i}\leq \frac{1}{2} \Ni{u_i}.
	\end{align*}
	Thus,$$\Ni{u_i}-2\sum_{\substack{j=1 \\ j\neq i}}^q \int_{\Omega} \beta_{ij}|u_i|^3|u_j|^3 \geq   \frac{1}{2} \Ni{u_i}  \geq \frac{1}{2}SC_1^{\frac{1}{3}}.$$	
Therefore, $B_q(\mathbf{u})$ is strictly diagonally dominant, and so $B_q(\mathbf{u})$ is positive definite. It follows that there exists a constant $C>0$ such that
			\begin{equation}
				\begin{aligned}
					f_q(t_1,...,t_q)&=\frac{1}{2}\sum_{i=1}^{q} t_i^2\left\|u_i \right\|_i^2-\frac{1}{6}\sum_{i,j=1}^{q} t_i^3t_j^3\int_{\Omega}\beta_{ij} \abs{u_i}^3\abs{u_j}^3  \\&\leq \frac{1}{2}\sum_{i=1}^{q} t_i^2\left\|u_i \right\|_i^2-\frac{C}{6}\sum_{i=1}^{q} t_i^6 \to -\infty ,\quad \text{ as } |\mathbf{t}|\to +\infty,
				\end{aligned}
			\end{equation}
			which implies that  $f_q(t_1,...,t_q)$ has a global maximum in $\overline{(\R^+)^q}$. Here, $\mathbf{t}=(t_1,...,t_q).$ Similar to the proof of Step2 in proposition \ref{Energy-comparing 1}, we can get that the global maximum point of $ f_q(x_1,...,x_q) $ can not belong to  $\partial \overline{(\R^+)^q}$, which implies that the global maximum point of $f_q(x_1,...,x_q)$ is a interior point  in $ (\R^+)^q $. Therefore, the global maximum point of $f_q$ is a critical point. Next, we will show that $f_q$ has a unique critical point.
			
		For convenience of calculations, we consider
		\begin{equation}
			\tilde{f_q}(t_1,....,t_q)=\frac{1}{2}\sum_{i=1}^{q} t_i^{\frac{2}{3}}\left\|u_i \right\|_i^2-\frac{1}{6}\sum_{i,j=1}^{q} t_it_j\int_{\Omega}\beta_{ij} \abs{u_i}^3\abs{u_j}^3 .
		\end{equation}
	By a direct calculation,
	\begin{equation}
		\begin{aligned}
			&\frac{\partial \tilde{f_q}}{\partial t_i}(t_1,...,t_q)=\frac{1}{3}t_i^{-\frac{1}{3}}\Ni{u_i}-\frac{1}{3}\sum_{j=1}^q t_j\int_{\Omega}\beta_{ij} \abs{u_i}^3\abs{u_j}^3 ,\quad 1\leq i \leq q,\\
			&\frac{\partial^2 \tilde{f_q}}{\partial t_i^2}(t_1,...,t_q)=-\frac{1}{9}t_i^{-\frac{4}{3}}\Ni{u_i}-\frac{1}{3}\int_{\Omega}\beta_{ii} \abs{u_i}^6 , \quad 1\leq i \leq q,\\
			&\frac{\partial^2 \tilde{f_q}}{\partial t_i \partial t_j}(t_1,...,t_q)=-\frac{1}{3}\int_{\Omega}\beta_{ij} \abs{u_i}^3\abs{u_j}^3 , \quad 1\leq i,j \leq q, \ i \neq j,
		\end{aligned}
	\end{equation}
Thus the Hessian matrix of $\tilde{f_q}$ is
\begin{equation}
	\begin{aligned}
		H(\tilde{f_q})&=-\frac{1}{9}\sbr{\begin{matrix}
				t_1^{-\frac{4}{3}}\Ni{u_1}&&\\
				&\ddots &\\
				& &t_q^{-\frac{4}{3}}\Ni{u_q}
		\end{matrix}}-\frac{1}{3}\sbr{\begin{matrix}
				b_{11}(\mathbf{u_q}) &\cdots &b_{1q}(\mathbf{u_q})\\
				\vdots&\ddots &\vdots\\
				b_{q1}(\mathbf{u_q}) &\cdots&b_{qq}(\mathbf{u_q})
		\end{matrix}}\\
	&=: -\frac{1}{9}B(\mathbf{t})-\frac{1}{3}B_q(\mathbf{u_q}),
	\end{aligned}	
\end{equation}
where $ \mathbf{t}=(t_1,...,t_q) \in (\R^+)^q $ and  $b_{ij}(\mathbf{u})$ is defined in \eqref{Bqmatrix}.
We already know the matrix $ B_q(\mathbf{u_q}) $ is positive definite  and it is easy to see the matrix $B(\mathbf{t})$  is also positive definite, 
thus the Hessian matrix of  $\tilde{f_q}$ is negative definite,
which implies that $\tilde{f_q}$ has a unique critical point. Therefore, the critical point must be the global maximum point. Notice that $ \tilde{f_q}(t_1^3,....,t_q^3) = f_q(t_1,...,t_q)$,  thus $f_q$ has a unique critical point and the critical point must be the global maximum point.
Since $\mathbf{u_q}\in \mathcal{N}_{1,\ldots q}$, then by a direct calculation we have
$$\cfrac{\partial f_q}{\partial t_i}(\mathbf{1})=\Ni{u_i}-\sum_{j=1 }^q \int_{\Omega} \beta_{ij} |u_i|^3|u_j|^3=0 \quad \text{ for every } 1\leq   i \leq q,
$$
which implies that $\mathbf{1}=(1,...,1)\in (\R^+)^q$ is a critical point. As a consequence, $\mathbf{1}=(1,...,1)$  is  a maximum point of $f_q$. In other words,
\begin{equation}
	 \max_{t_1,...,t_q>0}f_q(t_1,...,t_q)=f(1,...,1)=\mathcal{C}_{1,\ldots q}.
\end{equation}
This completes the proof.
		\end{proof}

	\begin{lemma} \label{esist of t_1,t_d}
					Assume that there holds $$\beta_{ii} > 0 \quad \forall i=1,2,...,d, \quad 0<\beta_{ij} < K_3 \quad  \forall i,j = 1,2,...,d, i \neq j.$$
Given $1\leq q  \leq d-1 $, if $\mathcal{C}_{1,\ldots q}$ is attained by $\mathbf{u}_q=(u_1,...,u_q) \in \mathcal{N}_{1,\ldots q}$, then there exists $\widetilde{t}_i>0, i=1,\ldots, d$, such that $$(\widetilde{t}_1u_1,...,\widetilde{t}_qu_q,\widetilde{t}_{q+1}\omega_{q+1},...,\widetilde{t}_d\omega_d) \in \mathcal{N},$$
where $\omega_i$ is a least energy positive solution of \eqref{B-N}.
				\end{lemma}
				
				\begin{proof}
					Let $\mathbf{v}=(v_1,...,v_d)=(u_1,...,u_q,\omega_{q+1},...,\omega_d) $ and
					\begin{equation}
						\varPhi(t_1,...,t_d) = J(t_1v_1,...,t_dv_d)=\frac{1}{2}\sum_{i=1}^{q} t_i^2\left\|v_i \right\|_i^2-\frac{1}{6} L_t M(\mathbf{v}) L_t^T
					\end{equation}
					where $L_t=(t_1^3,...,t_d^3)$ is a vector in $\R^d$ and $M(\mathbf{v})=\sbr{M_{ij}(\mathbf{v})}_{d\times d}$  is a symmetric matrix with
					\begin{equation}
						\begin{aligned}
							&M_{ij}(\mathbf{v})=\int_{\Omega}\beta_{ij}|u_i|^3|u_j|^3, \quad \text{ for every } 1\leq i,j \leq q\\
							&M_{ij}(\mathbf{v})=\int_{\Omega}\beta_{ij}|u_i|^3|\omega_j|^3, \quad \text{ for every } 1 \leq\  i\  \leq q,\  q+1 \leq j \leq d\\
							&M_{ij}(\mathbf{v})=\int_{\Omega}\beta_{ij}|\omega_i|^3|\omega_j|^3,\quad \text{ for every } q+1\leq i,j \leq d
						\end{aligned}
					\end{equation}
We will show that the matrix $M(\mathbf{v})$ is strictly diagonally dominant. We separate the proof into two cases.
					
For the case $1\leq i \leq q :$ we want to show that $$\int_{\Omega} \beta_{ii} |u_i|^6 -\sum_{\substack{j=1 \\ j\neq i}}^q \int_{\Omega} \beta_{ij}|u_i|^3|u_j|^3-\sum_{j=q+1}^d\int_{\Omega}\beta_{ij}|u_i|^3|\omega_j|^3 >0.$$
					In fact, by Lemma \ref{subsystem-maximum} we know
					$$\int_{\Omega} \beta_{ii} |u_i|^6 -\sum_{\substack{j=1 \\ j\neq i}}^q \int_{\Omega} \beta_{ij}|u_i|^3|u_j|^3 \geq \frac{1}{2}SC_1^{\frac{1}{3}}.$$
					Moreover, under the assumptions of $\beta_{ij}$ and the definition of $m_i$, we have
					\begin{equation}
						\begin{aligned}
							\sum_{j=q+1}^d\int_{\Omega}\beta_{ij}|u_i|^3|\omega_j|^3 &\leq K_3 \sum_{j=q+1}^d\sbr{\int_{\Omega}|u_i|^6}^{\frac{1}{2}}\sbr{\int_{\Omega} |\omega_j|^6}^{\frac{1}{2}}\\&\leq  \frac{K_3}{S^{\frac{3}{2}}}\sum_{j=q+1}^d\left\|u_i \right\|_i^3 \sqrt{\frac{3m_j}{\beta_{jj}}}\\&\leq \frac{K_3}{S^{\frac{3}{2}}} (6\overline{C})^{\frac{3}{2}} \sbr{\sum\limits_{j=1}^d\sqrt{\frac{3m_j}{\beta_{jj}}}} \leq \frac{1}{4}SC_1^{\frac{1}{3}},
						\end{aligned}
					\end{equation}
					which implies that
					\begin{equation} \label{15}
						\int_{\Omega} \beta_{ii} |u_i|^6 -\sum_{\substack{j=1 \\ j\neq i}}^q \int_{\Omega} \beta_{ij}|u_i|^3|u_j|^3-\sum_{j=q+1}^d\int_{\Omega}\beta_{ij}|u_i|^3|\omega_j|^3 \geq  \frac{1}{4}SC_1^{\frac{1}{3}}> 0.
					\end{equation}
				
					For the case $q+1 \leq i \leq d :$ We want to show $$ \int_{\Omega} \beta_{ii} |\omega_i|^6 -\sum_{j=1 }^q \int_{\Omega} \beta_{ij}|\omega_i|^3|u_j|^3-\sum_{\substack{j=q+1 \\ j\neq i}}^d\int_{\Omega}\beta_{ij}|\omega_i|^3|\omega_j|^3 >0. $$
					By a direct calculation, we have
					\begin{equation} \label{16}
						\begin{aligned}
							&\int_{\Omega} \beta_{ii} |\omega_i|^6 -\sum_{j=1 }^q \int_{\Omega} \beta_{ij}|\omega_i|^3|u_j|^3-\sum_{\substack{j=q+1 \\ j\neq i}}^d\int_{\Omega}\beta_{ij}|\omega_i|^3|\omega_j|^3\\
							& \geq 3m_i-K_3\sum_{j=1 }^q\sbr{\int_{\Omega} |\omega_i|^6}^{\frac{1}{2}}\sbr{\int_{\Omega} |u_j|^6}^{\frac{1}{2}}-K_3\sum_{\substack{j=q+1 \\ j\neq i}}^d \sbr{\int_{\Omega} |\omega_i|^6}^{\frac{1}{2}} \sbr{\int_{\Omega} |\omega_j|^6}^{\frac{1}{2}}\\
							&\geq 3m_i- K_3 \sqrt{\frac{3m_i}{\beta_{ii}}} \sum_{j=1 }^q\cfrac{\left\|u_j \right\|_j^3}{S^{\frac{3}{2}}}-K_3\sum_{\substack{j=q+1 \\ j\neq i}}^d \sqrt{\frac{3m_i}{\beta_{ii}}}\sqrt{\frac{3m_j}{\beta_{jj}}}\\
							&\geq 3m_i -K_3\sqrt{\frac{3m_i}{\beta_{ii}}}\mbr{\sbr{\cfrac{6\overline{C}}{S}}^{\frac{3}{2}}+\sum_{j=1}^d\sqrt{\frac{3m_j}{\beta_{jj}}}}\geq \frac{3}{2} m_i.
						\end{aligned}
					\end{equation}
We deduce from \eqref{15}) and \eqref{16} that $M(\mathbf{v})$ is strictly diagonally dominant, then $M(\mathbf{v})$ is positive definite. Then there exists $C>0$ such that
\begin{equation}
\begin{aligned}
\varPhi(t_1,...,t_d)&=\frac{1}{2}\sum_{i=1}^{q} t_i^2\left\|v_i \right\|_i^2-\frac{1}{6} L_t M(\mathbf{v}) L_t^T \\& \leq \frac{1}{2}\sum_{i=1}^{q} \sbr{t_i^2\left\|v_i \right\|_i^2 -\frac{C}{6} t_i^6} \to -\infty \quad \text{ as } |\mathbf{t}| \to + \infty.
\end{aligned}
\end{equation}
Therefore, $\varPhi(t_1,...,t_d)$ has a global maximum $(\widetilde{t}_1,...,\widetilde{t}_d)$ in $\overline{(\R^+)^d}$. By a simliar argument as used in Lemma $\ref{subsystem-maximum}$ Step 1, the global maximum point $(\widetilde{t}_1,...,\widetilde{t}_d)$  can  not belong to $\partial \overline{(\R^+)^d}$, and it must be a critical point.
Therefore, $(\widetilde{t}_1u_1,...,\widetilde{t}_qu_q,\widetilde{t}_{q+1}\omega_{q+1},...,\widetilde{t}_d\omega_d) \in \mathcal{N}.$
				\end{proof}
				
\begin{proof}[\bf Proof of proposition \ref{energy comparing c_q}:]
				Without loss of generality, we prove that
				\begin{equation}
					\mathcal{C} \leq \mathcal{C}_{1,\ldots,q}+\sum_{i=q+1}^d m_i.
				\end{equation}
Assume that $\mathcal{C}_{1,\ldots,q}$ is achieved by $\mathbf{u}_q=(u_1,\cdots,u_q)$.		
By Lemma \ref{esist of t_1,t_d} there exists  $(\widetilde{t}_1,...,\widetilde{t}_d)$ such that $(\widetilde{t}_1u_1,...,\widetilde{t}_qu_q,\widetilde{t}_{q+1}\omega_{q+1},...,\widetilde{t}_d\omega_d) \in \mathcal{N}.$
Note that $\beta_{ij}>0$ for any $i\neq j$,	then by a direct calculation we have
				\begin{equation}\label{40}
					\begin{aligned}
J(\widetilde{t}_1u_1,...,\widetilde{t}_qu_q,\widetilde{t}_{q+1}\omega_{q+1},...,\widetilde{t}_d\omega_d)
&\leq \frac{1}{2}\sum_{i=1}^q \widetilde{t_i}^2\Ni{u_i}-\frac{1}{6} \sum_{i,j=1}^q\widetilde{t_i}^3\widetilde{t_j}^3\int_{\Omega}\beta_{ij}|u_i|^3|u_j|^3\\
& \quad +\frac{1}{2}\sum_{i=q+1}^d \widetilde{t_i}^2\Ni{\omega_i}-\frac{1}{6} \sum_{i=q+1}^d\widetilde{t_i}^6\int_{\Omega}\beta_{ii}|\omega_i|^6 \\
&=:f(\widetilde{t}_1,...,\widetilde{t}_q)+ g(\widetilde{t}_{q+1},...,\widetilde{t}_d),
\end{aligned}	
\end{equation}
where $f(t_1,...,t_q)$ is defined in \eqref{defi of f}and
\begin{equation}
g(t_{q+1},...,t_d):=\frac{1}{2}\sum\limits_{i=q+1}^d t_i^2\Ni{\omega_i}-\frac{1}{6} \sum\limits_{i=q+1}^dt_i^6\int_{\Omega}\beta_{ij}|\omega_i|^6.
\end{equation}
Notice that$\Ni{\omega_i}=\int_{\Omega}\beta_{ij}|\omega_i|^6=3m_i$, it is easy to show that
\begin{equation} \label{41}
g(\widetilde{t}_{q+1},...,\widetilde{t}_d)\leq \max_{t_{q+1},...,t_d>0}g(t_{q+1},...,t_d) =\sum\limits_{i=q+1}^d m_i.
\end{equation}
By Lemma \ref{subsystem-maximum} we get that
\begin{equation}\label{42}				f(\widetilde{t}_1,...,\widetilde{t}_q)\leq\max_{t_1,...,t_q>0}f(t_1,...,t_q)=f(1,...,1)=\mathcal{C}_{1,\ldots,q}.
\end{equation}
				We deduce from \eqref{40}), \eqref{41} and \eqref{42} that
				\begin{equation}\begin{aligned}
						\mathcal{C} \leq J(\widetilde{t}_1u_1,...,\widetilde{t}_qu_q,\widetilde{t}_{q+1}\omega_{q+1},...,\widetilde{t}_d\omega_d) \leq f(\widetilde{t}_1,...,\widetilde{t}_q)+ g(\widetilde{t}_{q+1},...,\widetilde{t}_d)\leq \mathcal{C}_{1,\ldots,q}+\sum_{i=q+1}^dm_i.
					\end{aligned}	
				\end{equation}
			This completes the proof of Proposition \ref{energy comparing c_q}.
				 \end{proof}

\subsection{Proof of Theorem \ref{existence-1}}\label{Sec 2.3}
In this subsection, we present the proof of Theorem \ref{existence-1}. Recall that $ m_i<\frac{1}{3}\beta_{ii}^{-\frac{1}{2}} \widetilde{S}^{\frac{3}{2}}$ (see \eqref{Energy-BN}) for every $1 \leq i \leq d.$ Set				
\begin{equation} \label{3.6}
	\delta=\frac{1}{2}\min_{1\leq i \leq d}\lbr{\beta_{ii}^{-1}\widetilde{S}^3- (3m_i)^2}>0,
\end{equation}
then we have
\begin{equation} \label{3.7}
	(3m_i)^2 <\beta_{ii}^{-1}\widetilde{S}^{\frac{3}{2}}-\delta, \quad 1\leq i \leq d.
\end{equation}			
Denote
\begin{equation}\label{defi of K}
K_4=\min_{1\leq i\leq d}\left\{\frac{\beta_{ii}S^3}{(6\overline{C})^2\widetilde{S}^3}\delta\right\}  \ \text{ and }\ K=\min\lbr{ K_1, K_2, K_3, K_4},
\end{equation}
where $K_1$ is defined in \eqref{K_1}, $K_2$ is defined in \eqref{ defi of K_2}, $K_3$ is defined in \eqref{defi of K_3},  $\delta$ is fixed in \eqref{3.6}.
From now on, we assume that $\beta_{ij}$ satisfies $0<\beta_{ij} < K$ for any $i \neq j.$				
	
\begin{proof}[\bf Conclusion of the proof of Theorem \ref{existence-1}]
We will proceed by mathematical induction on the number of the equations in the subsystem. Set $|I|=M$, that is $M$ the number of the equations in the subsystem, and $M=1,\ldots, d$.

When $M=1$, system \eqref{mainequ} reduces to the following problem
\begin{equation*}
-\Delta u+\lambda_i=\beta_{ii}|u|^4u, \quad u\in H_0^1(\Omega),
\end{equation*}
and by \cite{Brezis-Nirenberg1983} we see that Theorem \ref{existence-1} is true.

We suppose by induction hypothesis that Theorem \ref{existence-1} holds true for every level $\mathcal{C}_I$ with $|I|\leq M$  for some $1\leq M\leq d-1$. We need prove Theorem \ref{existence-1} for $\mathcal{C}_I$ with $|I|=M+1$. Without loss of generality, we will present the proof for $I=\{1,\ldots,M+1\}$.  By induction hypothesis we know that Proposition \ref{energy comparing c_q} is true for $\mathcal{C}_I$.
By Lemma \ref{exist of ps sequence}, there exists a sequence $\lbr{\mathbf{u}_n}\subset \mathcal{N}_I$ satisfying $$\lim_{n \to \infty}J_I(\mathbf{u}_n)=\mathcal{C}_I ,\quad \lim_{n \to \infty}J^{\prime}_I(\mathbf{u}_n)=0,$$ then $\lbr{u_{i,n}}$ is uniformly bounded in $H_0^1{\sbr{\Omega}}$, $i=1,2,...,M+1$. Passing to subsequence, we may assume that
				\begin{equation}\label{convergence1}
					\begin{aligned}
						u_{i,n}\rightharpoonup u_i  \text{ weakly in }H_0^1{\sbr{\Omega}},\quad
						u_{i,n} \to u_i   \text{ strongly in }  L^2(\Omega).
					\end{aligned}
				\end{equation}
It is standard to see that $J^{\prime}_I(\mathbf{u})=0$ and
				\begin{equation} \label{23}
					\left\|u_i \right\|_i^2=\sum_{j=1}^{M+1} \int_{\Omega}\beta_{ij} \abs{u_i}^3\abs{u_j}^3  \quad \text{for every } i=1,2,..,M+1.
				\end{equation}	
				Denote $\sigma_{i,n}=u_{i,n}-u_i$, $i=1,2,...,M+1$, and so
				$$\sigma_{i,n}\rightharpoonup 0  \text{ weakly in }H_0^1{\sbr{\Omega}}.$$
				We deduce from \eqref{convergence1} that
				\begin{equation}\label{21}
					\int_{\Omega} |\nabla u_{i,n}|^2 =\int_{\Omega} |\nabla \sigma_{i,n}|^2 +\int_{\Omega} |\nabla u_{i}|^2  +o(1),
				\end{equation}
				and by lemma \ref{brezis-lieb lemma} we have
				\begin{equation}\label{22}
					\int_{\Omega} |u_{i,n}|^3|u_{j,n}|^3=\int_{\Omega}|\sigma_{i,n}|^3|\sigma_{j,n}|^3 +	\int_{\Omega} |u_{i}|^3|u_{j}|^3+o(1).
				\end{equation}
				By \eqref{21} and \eqref{22} we have
				\begin{equation}
					J_I(\mathbf{u}_n)=J_I(\mathbf{u})+\frac{1}{3}\sum_{i=1}^{M+1} \int_{\Omega} |\nabla \sigma_{i,n}|^2+o(1).
				\end{equation}
Passing to subsequence, we may assume that
				\begin{equation}
					\lim_{n \to \infty} \int_{\Omega} |\nabla \sigma_{i,n}|^2  =k_i \geq 0, \quad i=1,2,...,M+1.
				\end{equation}
				Thus,
				\begin{equation} \label{27}
					0\leq J_I(\mathbf{u}) \leq J_I(\mathbf{u}) + \frac{1}{3}\sum_{i=1}^{M+1} k_i =\lim_{n \to \infty}	J_I(\mathbf{u}_n) =\mathcal{C}_I .
				\end{equation}
				Next, we will show that all $u_i \not \equiv 0, 1\leq i\leq M+1$ by using a contradiction argument.
				
				{\bf Case 1:}  $u_i \equiv 0$ for every $  1\leq i\leq M+1 $.
				
Firstly, we claim that $k_i>0$, $i=1,2,...,M+1$. By contradiction, without loss of generality, we assume that $k_1=0$, notice that $\sigma_{1,n}=u_{1,n}$, then we know that $\sigma_{1,n} \to 0$ strongly in $H_0^1(\Omega)$  and $u_{1,n} \to 0$ strongly in $H_0^1(\Omega)$. Hence, by Sobolev inequality
				we have $$\lim_{n \to \infty}\int_{\Omega} |u_{1,n}|^6 =0.$$
				On the other hand, by Lemma \ref{L^6norm-estimate}, we see that
				$$\lim_{n \to \infty}\int_{\Omega} |u_{1,n}|^6   \geq C_1>0,$$
				which is a contradiction. Therefore, $k_i>0$, $i=1,2,...,M+1$. Notice that $J_I(\mathbf{u}_n)\leq2 \mathcal{C}_I\leq2\overline{C}$ for $n$ large enough, thus
				\begin{equation*}
					\int_{\Omega}|\nabla \sigma_{i,n}|^2 \leq \sum_{j=1}^{M+1}\int_{\Omega}|\nabla \sigma_{j,n}|^2+3J_I(\mathbf{u})+o(1)=3J_I(\mathbf{u}_n)\leq 6\overline{C}.
				\end{equation*}
Hence,
				\begin{equation} \label{3.20}
					0<k_i\leq 6\overline{C}.
				\end{equation}
 				 Since $\mathbf{u}_n \in \mathcal{N}_I$ and $\sigma_{i,n}=u_{i,n}$, then we have $\sum\limits_{i=1}^{M+1}\Ni{\sigma_{i,n}}=\sum\limits_{i=1}^{M+1}\Ni{u_{i,n}} \leq  6\overline{C}$. Therefore,
\begin{equation}
\begin{aligned}
\int_{\Omega} |\nabla \sigma_{i,n}|^2  &=\beta_{ii}\int_{\Omega} |\sigma_{i,n}|^6  +\sum_{\substack{j=1 \\ j\neq i}}^{M+1}\beta_{ij}\int_{\Omega} |\sigma_{i,n}|^3|\sigma_{j,n}|^3   \\
						&\leq \beta_{ii} \widetilde{S}^{-3} \sbr{	\int_{\Omega} |\nabla \sigma_{i,n}|^2}^3+ K\sum_{\substack{j=1 \\ j\neq i}}^{M+1} \sbr{\int_{\Omega} |\sigma_{i,n}|^6 }^{\frac{1}{2}}\sbr{\int_{\Omega} |\sigma_{j,n}|^6 }^{\frac{1}{2}}\\
						&\leq \beta_{ii} \widetilde{S}^{-3} \sbr{	\int_{\Omega} |\nabla \sigma_{i,n}|^2}^3+	KS^{-3}\left\|\sigma_{i,n} \right\|_i^3\sum_{\substack{j=1 \\ j\neq i}}^{M+1}   \left\|\sigma_{j,n} \right\|_j^3\\
						&\leq \beta_{ii} \widetilde{S}^{-3} \sbr{	\int_{\Omega} |\nabla \sigma_{i,n}|^2}^3+K S^{-3} (6\overline{C})^{\frac{3}{2}}\sbr{\int_{\Omega} |\nabla \sigma_{i,n}|^2  +o(1)}^{\frac{3}{2}}.\\
					\end{aligned}
				\end{equation}
				
				Let $n\to \infty$, we have
				\begin{equation}\label{2.27}
					k_i\leq \beta_{ii} \widetilde{S}^{-3}k_i^3+KS^{-3}(6\overline{C})^{\frac{3}{2}}k_i^{\frac{3}{2}}.				\end{equation}
			Combining this with \eqref{3.20}, we get
				\begin{equation}
					1\leq  \beta_{ii} \widetilde{S}^{-3}k_i^{2}+KS^{-3}(6\overline{C})^2.
				\end{equation}
			Then by the definition of $K$, $K_4$ and \eqref{3.7} we get
			\begin{equation}
				\begin{aligned}
					k_i^2\geq \beta_{ii}^{-1} \widetilde{S}^{3} -K \cfrac{\widetilde{S}^3 (6\overline{C})^2}{\beta_{ii}S^3}\geq \beta_{ii}^{-1} \widetilde{S}^{3} -\delta > (3m_i)^2,
				\end{aligned}
			\end{equation}
		which implies
		\begin{equation}
			k_i > 3m_i.
		\end{equation}
				By proposition \ref{Energy-comparing 1} and \eqref{27}  we have
				\begin{equation}
					\sum_{i=1}^{M+1} m_i\geq \mathcal{C}_I=\lim_{n \to \infty}J_I(\mathbf{u}_n)= J_I(\mathbf{u})+ \frac{1}{3}\sum_{i=1}^{M+1} k_i= \frac{1}{3}\sum_{i=1}^{M+1} k_i>\sum_{i=1}^{M+1} m_i,
				\end{equation}
				that is a contradiction. Therefore, Case 1 is impossible.
				
				{\bf Case 2:} Only one component of $\mathbf{u}$ is not zero.
				
				Without loss of generality, we assume that $u_1\not \equiv 0$, and $u_i\equiv 0$, $2\leq i \leq p+1$. Similarly to Case 1, we can prove that  $k_i>3 m_i >0$  for every $2\leq i \leq M+1$.
				Notice that $ (u_1,0,...,0) $ is a solution of \eqref{mainequ}, then $J(u_1,0,...,0) \geq  m_1$.
Combining this with Proposition \ref{Energy-comparing 1} and \eqref{27} we know that
				\begin{equation}
					\sum_{i=1}^{M+1} m_i\geq \mathcal{C}_I=\lim_{n \to \infty}J_I(\mathbf{u}_n)= J_I(u_1,0,...,0)+ \frac{1}{3}\sum_{i=1}^{M+1}k_i \geq m_1+ \frac{1}{3}\sum_{i=2}^{M+1}k_i >\sum_{i=1}^{M+1} m_i,
				\end{equation}
				that is a contradiction. Therefore, Case 2 is impossible.
				
				{\bf Case 3:} There are $q$ components of $\mathbf{u}$ that are not zero, $2\leq q \leq M$.
				
				Without of loss generality, we may assume that $u_1,...,u_q\not\equiv 0$, and $u_{q+1},...,u_{M+1}\equiv0$. Similarly to Case 1, we have $k_i>3m_i$, $q+1 \leq i \leq M+1$.
			Note that $ (u_1,u_2,...,u_q,0,,...,0) $ is a solution of subsystem  and $(u_1,u_2,...,u_q) \in \mathcal{N}_{1,\ldots q}$, then $J_I(\mathbf{u})\geq \mathcal{C}_{1,\ldots q}$.
				Combining this with  Proposition \ref{energy comparing c_q} and \eqref{27}, we have
				\begin{equation}\begin{aligned}
						\mathcal{C}_{1,\ldots q}+\sum\limits_{i=q+1}^{M+1} m_i  \geq \mathcal{C}_I =\lim_{n \to \infty} J_I(\mathbf{u}_n)= J_I(\mathbf{u})+\frac{1}{3}\sum_{i=1}^{M+1}k_i>\mathcal{C}_{1,\ldots q}+\sum_{i=q+1}^{M+1}m_i,
					\end{aligned}	
				\end{equation}
that is a contradiction. Therefore, Case 3 is impossible.
				
Since Case 1, Case 2 and Case 3 are impossible, then we get that all components of $\mathbf{u}=(u_1,...,u_{M+1})$ are not zero. Therefore $\mathbf{u}\in \mathcal{N}_I$. Combining this with \eqref{27} we see that
\begin{equation}
\mathcal{C}_I\leq J_I(\mathbf{u}) \leq  J_I(\mathbf{u})+\frac{1}{3} \sum_{i=1}^{M+1} k_i =\lim_{n \to \infty} J_I(\mathbf{u}_n)=\mathcal{C}_I,
\end{equation}
which yields that $J_I(\mathbf{u})=\mathcal{C}_I$. Obviously, $$\mathbf{\widehat{u}}=(|u_1|,...,|u_{M+1}|) \in \mathcal{N}_I\  \text{ and }\  J_I(\mathbf{\widehat{u}})=\mathcal{C}_I.$$
It follows from Lemma \ref{energy uniformly estimate} and \ref{diagonally dominant} that $\mathbf{\widehat{u}}$ is a nonnegative critical point of $J_I$, and $(|u_1|,...,|u_{M+1}|)$ is a nonnegative solution of system \eqref{mainequ}. By the maximum principle, we know that $ |u_i|>0$ in $\Omega$, $1\leq i \leq M+1$. Therefore, $\mathbf{\widehat{u}}$ is a least energy positive solution of subsystem \eqref{subsystem1} with $I=\lbr{1,\cdots,M+1}$. We proceed by repeating this step, then we obtain a least energy positive solution of subsystem \eqref{subsystem1} with $I=\lbr{1,\cdots,d}$.   This completes the proof.
\end{proof}

\newpage
\section{Ground state solutions for the weakly cooperative case} \label{Sec 2.4}

In this section, we show the proof of Theorem \ref{existence-4}.

\subsection{limit system } \label{Sec 3.1}

Since the problem \eqref{mainequ} has a critical nonlinearity and critical coupling terms, the existence of nontrivial ground state solutions  of \eqref{mainequ}  strongly depend on the existence of the ground state solutions of the following limit system
\begin{equation} \label{limit system}
	\begin{cases}
		-\Delta u_i =\sum\limits_{j=1}^{d} \beta_{ij}|u_j|^{3}|u_i|u_i  \quad ~\text{ in } \R^3,\\
		u_i \in \mathcal{D}^{1,2}(\R^3) ,\quad i= 1,2,...,d,
	\end{cases}
\end{equation}
where $\mathcal{D}^{1,2}(\R^3)=\lbr{u\in L^2(\R^3): |\nabla u| \in L^2(\R^3)}$ with norm $\left\|u \right\|_{\mathcal{D}^{1,2}}:=\sbr{\int_{\R^3}|\nabla u|^2 }^{\frac{1}{2}} $.
Define $ \mathbb{D}:=\sbr{ \mathcal{D}^{1,2}(\R^3)}^d$ and $C^1$ functional $E:\mathbb{D} \to \R$ as follows

\begin{equation}
	E(\mathbf{u}):=\frac{1}{2} \sum_{i=1}^d \int_{\R^3}|\nabla u_i|^2  -\frac{1}{6} \sum_{i,j=1}^d \int_{\R^3} \beta_{ij} |u_i|^3|u_j|^3 .
\end{equation}
We consider the set
\begin{equation} \label{M'}
	\mathcal{M}^\prime=\lbr{\mathbf{u}\in \mathbb{D}\setminus \lbr{\mathbf{0}}:  \sum_{i=1}^d\int_{\R^3}|\nabla u_i|^2= \sum_{i,j=1}^d \int_{\R^3} \beta_{ij} |u_i|^3|u_j|^3 }.
\end{equation}
Then any nontrivial solution of \eqref{limit system} belongs to $\mathcal{M}^\prime$. We set
\begin{equation} \label{defi of B}
	\mathcal{B}:= \inf_{\mathbf{u}\in \mathcal{M}^\prime} E(\mathbf{u})=\inf_{\mathbf{u}\in \mathcal{M}^\prime} \frac{1}{3} \sum_{i=1}^d \int_{\R^3} |\nabla u_i|^2 .
\end{equation}
For $\varepsilon>0$ and $y \in \R^3$, we consider the Aubin-Talenti bubble (\cite{Aubin=JDG=1976},\cite{Talenti=AMPA=1976})  $U_{\varepsilon,y}\in\mathcal{D}^{1,2}(\R^3) $ defined by
\begin{equation}
	U_{\varepsilon,y}(x)= \cfrac{(3\varepsilon^2)^{\frac{1}{4}}}{\sbr{\varepsilon^2+|x-y|^2}^{\frac{1}{2}}}.
\end{equation}
Then $U_{\varepsilon,y}$ solves the equation
\begin{equation}
	-\Delta u = u^5 \text{ in } \R^3,
\end{equation}
and
\begin{equation}
	\int_{\R^3}|\nabla U_{\varepsilon,y}|^2  =\int_{\R^3}| U_{\varepsilon,y}|^6 = \widetilde{S}^{\frac{3}{2}},
\end{equation}
where $\widetilde{S}$ is the Sobolev best constant of $\mathcal{D}^{1,2}(\R^3)\hookrightarrow L^6(\R^3)$. Furthermore, $\lbr{U_{\varepsilon,y}: \varepsilon>0, y \in \R^3 }$ contains all positive solutions of the equation $-\Delta u = u^5 \text{ in } \R^3$. To simplify the notation, we denote
\begin{equation} \label{defi of U{varepsilon}}
	U_{\varepsilon}(x):=U_{\varepsilon,0}(x)= \cfrac{(3\varepsilon^2)^{\frac{1}{4}}}{\sbr{\varepsilon^2+|x|^2}^{\frac{1}{2}}}.
\end{equation}

Thanks to \cite{HeYang2018}, we can get the existence and classification results for ground state solutions of system \eqref{limit system}, which is used to prove the existence of ground state solution in the next subsection. Before proceeding, we need introduce some notations.

Consider the polynomial function $P:\R^d \to \R$ defined by
\begin{equation*}
	P(\mathbf{x})=\sum_{i,j=1}^d\beta_{ij}|x_i|^3|x_j|^3
\end{equation*}
and denote by $\mathcal{X}$ the set of solutions to the maximization problem
\begin{equation}\label{P-def}
	P(\boldsymbol{\tau})=\max_{|\mathbf{X}|=1}P(\mathbf{X})=P_{\max},  \ \boldsymbol{\tau}=\sbr{\tau_1,\cdots,\tau_d}, \ \abs{\boldsymbol{\tau}}=1.
\end{equation}

\begin{lemma} \label{solution of limlit system}
	Assume that \eqref{assumption -2} holds, then the level $\mathcal{B}$ is attained by a solution of system \eqref{limit system}. Moreover, any of such minimizers has the form $\sbr{\tau_1P_{\max}^{-\frac{1}{4}}U_{\varepsilon,y},\cdots,\tau_d P_{\max}^{-\frac{1}{4}}U_{\varepsilon,y}}$, where $\boldsymbol{\tau}=\sbr{\tau_1,\cdots,\tau_d} \in \mathcal{X}$ and $y \in \R$, $\varepsilon>0$.
\end{lemma}

\subsection{Proof of Theorem \ref{existence-4}} \label{Sec 3.2}

In this subsection, we start to prove Theorem \ref{existence-4}. Recall the Nehari manifold
\begin{equation*}
	\mathcal{M}=\lbr{\mathbf{u}\in \mathbb{H}_d:\  \mathbf{u}\neq \mathbf{0},\  \sum_{i=1}^d\Ni{u_i}=\sum_{i,j=1}^d\int_{\Omega}\beta_{ij}|u_i|^3|u_j|^3 },
\end{equation*}
and the level of $J$
\begin{equation*}
	\mathcal{A}=\inf\lbr{J(\mathbf{u}): \mathbf{u} \in \mathcal{M}},
\end{equation*}
which are defined in the Introduction.
Define
\begin{equation}
	\widetilde{\mathcal{C}}:=\inf_{\gamma\in \Gamma}\max_{t\in[0,1]}J(\gamma(t)),
\end{equation}
where $\Gamma=\lbr{\gamma\in C([0,1],\mathbb{H}_d): \gamma(0)=\mathbf{0},J(\gamma(1))<0}$.
It is easy to see that
\begin{equation}
	\widetilde{\mathcal{C}}=\inf_{\mathbf{u}\in \mathbb{H}_d \setminus \lbr{\mathbf{0}}} \max_{t>0}J(t\mathbf{u})=\inf_{\mathbf{u} \in \mathcal{M}} J(\mathbf{u})=\mathcal{A}.
\end{equation}

Next, we present an energy estimate for the level $\mathcal{A}$, which plays a critical role in showing that the limit of Palais-Smale sequence is not zero.
\begin{lemma} \label{GSS energy estimate 1}
Assume that \eqref{assumption-1} and \eqref{assumption -2} hold. Then we have $\mathcal{A}< \mathcal{B}$, where $\mathcal{B}$ is defined in \eqref{defi of B}.
\end{lemma}
\begin{proof}
Without loss of generality, we assume that $0\in \Omega$ and $B_{R_0}(0) $ is the largest ball contained in $\Omega$, then  we take $\lambda^*(\Omega)= \frac{\pi^2}{4R_0^2}$ and
		\begin{equation}
			\varphi(x)=\begin{cases}
				\cos\sbr{\frac{\pi |x|}{2R_0}},  &x \in B_{R_0}(0),\\
				0 ,& x \in \Omega \setminus B_{R_0}(0).
			\end{cases}
		\end{equation}
		 Set
		\begin{equation*}
			w_{\varepsilon}(x):=U_{\varepsilon}(x)\varphi(x).
		\end{equation*}
	and
		\begin{equation}
		\widetilde{V}_i^\varepsilon(x)=\varphi(x)V_i^\varepsilon(x)=\tau_iP_{\max}^{-\frac{1}{4}}w_{\varepsilon}(x).
	\end{equation}
		where $U_{\varepsilon}$ is defined in \eqref{defi of U{varepsilon}}. By a standard argument ( see Lemma 1.3 in \cite{Brezis-Nirenberg1983} ), we get
\begin{equation}\label{estimate1}
\int_{\Omega} |\nabla w_{\varepsilon}|^2 = \widetilde{S}^{\frac{3}{2}}+\frac{\sqrt{3}}{2R_0}\pi^3\varepsilon + O(\varepsilon^2),
\end{equation}
		\begin{equation}\label{estimate2}
			\int_{\Omega} |w_{\varepsilon}|^6  = \widetilde{S}^{\frac{3}{2}}+ O(\varepsilon^2),
		\end{equation}
		\begin{equation}\label{estimate3}
			\int_{\Omega} |w_{\varepsilon}|^2  = 2\sqrt{3}\pi\varepsilon R_0+ O(\varepsilon^2).
		\end{equation}
		
By Lemma \ref{solution of limlit system} we know that $\sbr{V_1^{\varepsilon},\cdots,V_d^{\varepsilon}}=\sbr{\tau_1P_{\max}^{-\frac{1}{4}}U_{\varepsilon},\cdots,\tau_d P_{\max}^{-\frac{1}{4}}U_{\varepsilon}}$ is a ground state solution of system \eqref{limit system}, where $P_{\max}$, $\tau_i$ are defined in \eqref{P-def} and $U_{\varepsilon}$ is defined in \eqref{defi of U{varepsilon}}. Then
		\begin{equation*}
			\mathcal{B}=E\sbr{V_1^{\varepsilon},\cdots,V_d^{\varepsilon}}= \frac{1}{3}\sum_{i,j=1}^d \int_{\R^3}\beta_{ij} |V_i^\varepsilon|^3|V_j^\varepsilon|^3  =\frac{1}{3}\sbr{\sum_{i,j=1}^d \beta_{ij} |\tau_i|^3|\tau_j|^3} P_{\max}^{-\frac{3}{2}}\int_{\R^3} |U_{\varepsilon}(x)|^6 =\frac{1}{3}P_{\max}^{-\frac{1}{2}}  \widetilde{S}^{\frac{3}{2}}.
		\end{equation*} 			
		
	On the other hand,	we deduce from $\abs{\boldsymbol{\tau}}=1$ that there must exist $1\leq i_0\leq d$ such that $\tau_{i_0}^2>0$.
Based on this fact, for $\varepsilon>0 $ small, by \eqref{estimate1}, \eqref{estimate2} and \eqref{estimate3} we have
\begin{equation}
	\begin{aligned}
		&\max_{t>0} J(t\widetilde{V}_1^\varepsilon(x),\cdots,t\widetilde{V}_d^\varepsilon(x))\\
		&=\frac{1}{3} \cfrac{ \sbr{\sum\limits_{i=1}^d \int_{\Omega} |\nabla \widetilde{V}_i^\varepsilon(x)|^2+\lambda_{i} |\widetilde{V}_i^\varepsilon(x)|^2 dx}^{\frac{3}{2}}}{\sbr{\sum\limits_{i,j=1}^d \int_{\Omega} \beta_{ij} |\widetilde{V}_i^\varepsilon(x)|^3|\widetilde{V}_j^\varepsilon(x)|^3 dx }^{\frac{1}{2}}}\\
		&=\frac{1}{3} \cfrac{\mbr{\sum\limits_{i=1}^d \tau_i^2 P_{\max}^{-\frac{1}{2}}\sbr{\int_{\Omega} |\nabla w_{\varepsilon}(x)|^2 dx+\lambda_{i} \int_{\Omega} |w_{\varepsilon}(x)|^2 dx }}^{\frac{3}{2}}}{\sbr{P_{\max}^{-\frac{3}{2}}\sum\limits_{i,j=1}^d \beta_{ij}|\tau_i|^3|\tau_j|^3}^{\frac{1}{2}} \sbr{\int_{\Omega} |w_{\varepsilon}(x)|^6 dx}^{\frac{1}{2}}}\\
		&=\frac{1}{3} \cfrac{\mbr{\widetilde{S}^{\frac{3}{2}}+\sum\limits_{i=1}^d 2\sqrt{3}\pi\varepsilon R_0\tau_i^2 \sbr{ \lambda_{i}+\frac{\pi^2}{4R_0^2}}+ O(\varepsilon^2)}^{\frac{3}{2}}} {P_{\max}^{\frac{1}{2}} \sbr{\widetilde{S}^{\frac{3}{2}}+ O(\varepsilon^2)}^{\frac{1}{2}}}\\
		& < \frac{1}{3}P_{\max}^{-\frac{1}{2}}  \widetilde{S}^{\frac{3}{2}} = \mathcal{B}.
	\end{aligned}
\end{equation}
As a consequence,
\begin{equation}
	\mathcal{A}= \inf_{\mathbf{u}\in \mathbb{H}_d\setminus\lbr{\mathbf{0}}} \max_{t>0} J(t\mathbf{u}) \leq \max_{t>0} J(t\widetilde{V}_1^\varepsilon(x),\cdots,t\widetilde{V}_d^\varepsilon(x))< \mathcal{B}.
\end{equation}
This completes the proof.
	\end{proof}
			
\begin{proposition} \label{ground state achieve}
Suppose that $\beta_{ij}\geq 0$ for any $i\neq j$. Then $\mathcal{A}$ is attained on $\mathcal{M}$.
\end{proposition}
\begin{proof}
It is easy to see that the functional $J$ has a mountain pass structure, by the mountain pass theorem (see \cite{william=1996}), there exists $\lbr{\mathbf{u}_n} \subset \mathbb{H}_d$ such that
		\begin{equation*}
				\lim\limits_{n\to \infty}J(\mathbf{u}_n) = \mathcal{A}, \quad \lim_{n \to \infty} J^\prime(\mathbf{u}_n)= 0,
			\end{equation*}
			where $\mathbf{u}_n= \sbr{u_{1,n},\cdots,u_{d,n}}$. By a standard argument it is easy to see that $\lbr{\mathbf{u}_n} $ is bounded in $\mathbb{H}_d$. Up to subsequence, we may assume that
			\begin{equation} \label{4.36}
				\begin{aligned}
					&u_{i,n} \rightharpoonup u_i   \ \text{  weakly in } \ H_0^1(\Omega),
					&u_{i,n}  \rightarrow u_i   \ \text{  strongly in } \ L^2(\Omega).
				\end{aligned}
			\end{equation}
	It is standard to show $J^\prime(\mathbf{u})=0$. Set $\sigma_{i,n}= u_{i,n}-u_i$,  by Lemma \ref{brezis-lieb lemma} we have
			\begin{equation}
				\int_{\Omega}|u_{i,n}|^6 =\int_{\Omega}|\sigma_{i,n}|^6  + \int_{\Omega}|u_i|^6+o(1),
			\end{equation}
		and for $ i \neq j $,
		\begin{equation*}
				\int_{\Omega} |u_{i,n}|^3|u_{j,n}|^3=	\int_{\Omega}|\sigma_{i,n}|^3|\sigma_{j,n}|^3 +\int_{\Omega} |u_{i}|^3|u_{j}|^3 +o(1).
			\end{equation*}
			We deduce from \eqref{4.36} that
	\begin{equation} \label{4.38}
		\int_{\Omega} |\nabla u_{i,n}|^2 =\int_{\Omega} |\nabla \sigma_{i,n}|^2 +\int_{\Omega} |\nabla u_{i}|^2  +o(1).
	\end{equation}
	 Note that $\mathbf{u}_n\in \mathcal{M}$ and $J^\prime(\mathbf{u}_n) \to 0$, by a direct calculation we have
	 \begin{equation} \label{4.40}
	 	\int_{\Omega} |\nabla \sigma_{i,n}|^2 = \sum_{j=1}^d \int_{\Omega} \beta_{ij}|\sigma_{i,n}|^3|\sigma_{j,n}|^3 +o(1),
	 \end{equation}
	and
	\begin{equation} \label{4.39}
	J(\mathbf{u}_n)=J(\mathbf{u})+\frac{1}{3}\sum_{i=1}^d \int_{\Omega} |\nabla \sigma_{i,n}|^2 +o(1),
	\end{equation}
	which implies that $ \int_{\Omega} |\nabla \sigma_{i,n}|^2 $ is uniformly bounded for every $i=1,2,...,d$ and $n\in \N$. Passing to subsequence, we may assume that
	\begin{equation}
	\lim_{n \to \infty} \int_{\Omega} |\nabla \sigma_{i,n}|^2  =b_i \geq 0,\quad 1 \leq i \leq d.
	\end{equation}
	Thus,
	\begin{equation} \label{4.41}
	0\leq J(\mathbf{u}) \leq J(\mathbf{u}) + \frac{1}{3}\sum_{i=1}^d b_i =\lim_{n \to \infty}	J(\mathbf{u}_n) =\mathcal{A}  .
	\end{equation}

			Next, we will show that $\mathbf{u}\neq \mathbf{0}$. By using a contradiction argument we assume that  all components of $\mathbf{u}$ are zero, i.e.,  $u_i\equiv 0$, $ i=1,2,...,d $.		By \eqref{4.41}, we see that $b_1+b_2+\cdots+b_d=3\mathcal{A}>0$. Then we may assume that $(\sigma_{1,n},\cdots,\sigma_{d,n}) \neq \mathbf{0}$ for $n$ large. Recall the definition of $\mathcal{M}^\prime$ in \eqref{M'} and by \eqref{4.40}, it is easy to check that there exist $t_n$ such that $\sbr{t_n\sigma_{1,n},\cdots,t_n\sigma_{d,n}} \in \mathcal{M}^\prime$ and $t_n \to 1$ as $n \to \infty$. Then by \eqref{4.41}, we have
			\begin{equation*}
				\mathcal{A} =\frac{1}{3} \sum_{i=1}^d b_i= \lim_{n \to \infty} E\sbr{\sigma_{1,n},\cdots,\sigma_{d,n}} = \lim_{n \to \infty} E\sbr{t_n\sigma_{1,n},\cdots,t_n\sigma_{d,n}} \geq \mathcal{B},
			\end{equation*}
		which is a  contradiction to  Lemma \ref{GSS energy estimate 1}. Therefore, we have $\mathbf{u}\neq \mathbf{0}$. Since $J^\prime (\mathbf{u})=0$, we have  $\mathbf{u} \in \mathcal{M}$. Combing this with \eqref{4.41} we have
		\begin{equation*}
			\mathcal{A}=\inf_{\mathbf{v} \in \mathcal{M}} J(\mathbf{v})\leq J(\mathbf{u}) \leq \mathcal{A},
		\end{equation*}
	that is $ J(\mathbf{u}) = \mathcal{A} $. This completes the proof.
\end{proof}

The following proposition shows that the minimizer of $\mathcal{A}$ is semi-trivial for the weakly cooperative case.
\begin{proposition}\label{nonexistence-2}
If $\beta_{ij}\equiv b$, for any $i\neq j$, and
\begin{equation*}
0<b < 2^{\frac{3-d}{2}} \sqrt{\max_{1\leq i \leq d}\lbr{\beta_{ii}}\min_{1\leq i \leq d}\lbr{\beta_{ii}}},
\end{equation*}
then the system \eqref{mainequ} has no nontrivial ground state solutions.
\end{proposition}

\begin{proof}	
Based on Proposition \ref{ground state achieve}, we know that $\mathcal{A}$ can be attained by a  $\mathbf{u}=\sbr{u_1,\cdots,u_d}\in \mathcal{M}$, and $\mathbf{u}=\sbr{u_1,\cdots,u_d}$ is a ground state solution of system \eqref{mainequ}. Assume now that the ground state solution $\mathbf{u}$ is nontrivial.
Notice that
		\begin{equation}
			(tu_1,0,0,\cdots,0) \in \mathcal{M}\  \Leftrightarrow \  t^4=\cfrac{\left\| u_1\right\|_1^2 }{\beta_{11}|u_1|_6^6}\ ,
		\end{equation}
	and $\mathbf{u}$ is a ground state solution, we have
\begin{equation}\label{4.22}
	\begin{aligned}
		&J(u_1,...,u_d) \leq J(tu_1,0,\cdots,0) \Leftrightarrow \frac{1}{3} \sbr{\sum_{i=1}^d\beta_{ii}|u_i|_6^6 +b\sum_{\substack{i,j=1 \\ j\neq i}}^d |u_iu_j|_3^3} \leq \frac{1}{3}\frac{\left\| u_1\right\|_1^3 }{\sbr{\beta_{11}|u_1|_6^6}^{\frac{1}{2}}} \\   & \Leftrightarrow  \beta_{11}|u_1|_6^6 \sbr{\sum_{i=1}^d\beta_{ii}|u_i|_6^6 +b\sum_{\substack{i,j=1 \\ j\neq i}}^d |u_iu_j|_3^3}^2 \leq \sbr{\beta_{11}|u_1|_6^6 +b\sum_{\substack{j=2}}^d |u_1u_j|_3^3}^3.
	\end{aligned}
\end{equation}
	In this proof, to simplify the notations,  we take
	\begin{equation}
		A=\beta_{11}|u_1|_6^6, \ B=\sum\limits_{i=2}^d\beta_{ii}|u_i|_6^6, \  C=\sum\limits_{\substack{i,j=1 \\ j\neq i}}^d |u_iu_j|_3^3,\  D=\sum\limits_{\substack{j=2}}^d |u_1u_j|_3^3, \ E= \sum\limits_{\substack{i,j=2 \\ j\neq i}}^d |u_iu_j|_3^3.
	\end{equation}
	 Notice $C=2D+E$, then \eqref{4.22} is equivalent to
	\begin{equation} \label{4.24}
		A(A+B+2bD+bE)^2 \leq (A+bD)^3.
	\end{equation}
	By a direct calculation,  \eqref{4.24} implies that
	\begin{equation}
		AB^2+2A^2B \leq b^3D^3.
	\end{equation}
	Define
	$$a=\frac{2^{d-2}}{\beta_{11}}\max_{2\leq i \leq d}\lbr{\frac{1}{\beta_{ii}}}.$$
	Notice that
	\begin{equation}
		\begin{aligned}
			D^2 &=\sbr{\sum\limits_{\substack{j=2}}^d |u_1u_j|_3^3}^2 \leq \sbr{|u_1|_6^3\sum\limits_{\substack{j=2}}^d |u_j|_6^3}^2\leq 2^{d-2}|u_1|_6^6 \sum_{j=2}^d|u_j|_6^6 \leq  aAB.
		\end{aligned}	
	\end{equation}
	Thus
	\begin{equation}
		2\sqrt{2}A^{\frac{3}{2}}B^{\frac{3}{2}} \leq A^2B+2AB^2 \leq b^3D^3 \leq b^3a^{\frac{3}{2}}A^{\frac{3}{2}}B^{\frac{3}{2}}.
	\end{equation}
	From this inequality, we obtain that
	\begin{equation}
		b\geq \sqrt{2}a^{-\frac{1}{2}}=2^{\frac{3-d}{2}}\beta_{11}^\frac{1}{2}\min_{2\leq i \leq d} \lbr{\beta_{ii}}^{\frac{1}{2}}.
	\end{equation}
	By interchanging the roles of $\beta_{11}$ and $\beta_{jj}$, $j\geq 2$, we get that for all $j \in \lbr{1,\cdots,d}$,
	\begin{equation}
			b\geq \sqrt{2}a^{-\frac{1}{2}}=2^{\frac{3-d}{2}}\beta_{jj}^\frac{1}{2}\min_{\substack{1\leq i \leq d \\ i \neq j}} \lbr{\beta_{ii}}^{\frac{1}{2}}.
	\end{equation}
	In particular, 	
	\begin{equation}
		b\geq \sqrt{2}a^{-\frac{1}{2}}=2^{\frac{3-d}{2}}\max_{1\leq i \leq d}\lbr{\beta_{jj}}^\frac{1}{2}\min_{1\leq i \leq d } \lbr{\beta_{ii}}^{\frac{1}{2}}.
	\end{equation}
	Thus, if \begin{equation}
		b < \sqrt{2}a^{-\frac{1}{2}}=2^{\frac{3-d}{2}}\max_{1\leq i \leq d}\lbr{\beta_{jj}}^\frac{1}{2}\min_{1\leq i \leq d } \lbr{\beta_{ii}}^{\frac{1}{2}}
	\end{equation}
holds, then  the system \eqref{mainequ} has no nontrivial ground state solution.
\end{proof}

\begin{proof}[\bf Conclusion of the proof of Theorem \ref{existence-4}]
By Proposition \ref{ground state achieve} and Proposition \ref{nonexistence-2}, we know that Theorem \ref{existence-4} is true.
\end{proof}

\section{Existence for the purely competitive case}\label{Sec 4}

In this section, we consider the purely competitive case $\beta_{ij}\leq0$ and present the proof of theorem \ref{existence-4-1}. Recall the definitions of $\mathcal{N}_I$, $J_I$ and $\mathcal{C}_I$ in section \ref{Sec 2}, where $I\subseteq \{1,2,\ldots, d\}$.
To prove theorem \ref{existence-4-1}, we need the following several fundamental lemmas.

\begin{lemma}\label{L6 norm -2}
	Given $I \subseteq \lbr{1,2,...,d}$.	Assume that $\beta_{ij}\leq 0$ for all $i\neq j$, then there exist $C_3>0$ such that for any $\mathbf{u}_I\in \mathcal{N}_I$ there holds
	\begin{equation*}
		\int_{\Omega} |u_i|^6 \geq C_3, \quad i\in I.
	\end{equation*}
\end{lemma}
\begin{proof}
The proof is similar to that of \cite[Lemma 3.2]{Zou 2015}, so we omit it.
\end{proof}

\begin{lemma} \label{Free critical point}
	Given $I \subseteq \lbr{1,2,...,d}$,  $\mathcal{N}_I$ is a smooth manifold. Moreover, the constrained critical points of $J_I$ on $\mathcal{N}_I$ are free critical point of $J_I$.
\end{lemma}
\begin{proof}
	The proof of this lemma can be completed by the method analogous to that used in Lemma \ref{achieve and crtical point}. Recall the proof of Lemma \ref{achieve and crtical point}, the key point  is to show that $A_I(\mathbf{u})$ is strictly diagonally dominant. However, in the purely competitive case, the result is straightforward. Since $\beta_{ij}<0$ and $u_i\not \equiv 0$, we have
	\begin{equation}
		\begin{aligned}
				a_{ii}(\mathbf{u})-\sum_{\substack{j \in I , j\neq i}}|a_{ij}(\mathbf{u})|
				&=4\beta_{ii} |u_i|_6^6 +\sum_{\substack{j\in I , j\neq i}}  \beta_{ij}|u_iu_j|_3^3 -3\sum_{\substack{j\in I , j\neq i}}\left|\beta_{ij}\right||u_iu_j|_3^3 \\
				&=4\beta_{ii} |u_i|_6^6 +4\sum_{\substack{j\in I , j\neq i}}  \beta_{ij}|u_iu_j|_3^3=4 \Ni{u_i}>0 , \ i \in I,
		\end{aligned}
	\end{equation}
which implies that $A_I(\mathbf{u})$ is strictly diagonally dominant in the purely competitive case. Then using the same arguments as in the proof of Lemma \ref{achieve and crtical point}, we can easily carry out the proof of this lemma.
\end{proof}
\begin{lemma} \label{PS sequence}
	Given $I\subseteq \lbr{1,\cdots,d}$,  there exists a sequence $\lbr{\mathbf{u}_n}\subset \mathcal{N}_I$ satisfying $$\lim_{n \to \infty}J_I(\mathbf{u}_n)=\mathcal{C}_I ,\quad \lim_{n \to \infty}J_I^{\prime}(\mathbf{u}_n)=0 \text{ in } H^{-1}(\Omega).$$
\end{lemma}

\begin{proof}[\bf Proof]
The proof is similar to that of Lemma \ref{exist of ps sequence}, so we only sketch it.	
Recall the proof of Lemma \ref{exist of ps sequence}, the crucial step is to show  that the inequality \eqref{2.2} holds. Since $\beta_{ij}\leq0$ and $\mathbf{u}_n \in \mathcal{N}_I$, by Lemma \ref{L6 norm -2} we have
	\begin{equation} \label{4.1}
		\begin{aligned}
				&4\beta_{ii} |u_{i,n}|_6^6 +\sum_{\substack{j\in I, j\neq i}} \int_{\Omega} \beta_{ij}|u_{i,n}|^3|u_{j,n}|^3 -3\sum_{\substack{j\in I , j\neq i}}\left| \int_{\Omega} \beta_{ij}|u_{i,n}|^3|u_{j,n}|^3\right|\\ & = 4\Ni{u_{i,n}}\geq  4S|u_{i,n}|_6^2\geq 4SC_3^\frac{1}{3}>0  \ \text{ for every } i\in I.
		\end{aligned}
	\end{equation}
The remainder of the argument is analogous to that in Lemma \ref{exist of ps sequence}. 
\end{proof}

\begin{lemma} \label{achieve2}
	Given $I\subseteq \lbr{1,\cdots,d}$, if
$$
\mathcal{C}_I < \min \lbr{\mathcal{C}_\Gamma+\frac{1}{3}\sum_{i\in I \setminus \Gamma }\beta_{ii}^{-\frac{1}{2}}\widetilde{S}^{\frac{3}{2}}: \text{ for any } \Gamma \subsetneq I},
$$
then $\mathcal{C}_I$ is attained by $J_I$ on $\mathcal{N}_I$.
\end{lemma}
\begin{proof}
		Based on Lemma \ref{PS sequence}, there exists a sequence $\lbr{\mathbf{u}_n}\subset \mathcal{N}_I$ satisfying $$\lim_{n \to \infty}J_I(\mathbf{u}_n)=\mathcal{C}_I ,\quad \lim_{n \to \infty}J_I^{\prime}(\mathbf{u}_n)=0 \text{ in } H^{-1}(\Omega).$$
	Thus, $\lbr{\mathbf{u}_n}$ is bounded in $\sbr{H_0^1(\Omega)}^{|I|}$. So, after passing to subsequence, we may assume
	\begin{equation}
		\begin{aligned}
			&u_{i,n} \rightharpoonup u_i   \ \text{  weakly in } \ H_0^1(\Omega),
			&u_{i,n}  \rightarrow u_i   \ \text{  strongly in } \ L^2(\Omega).
		\end{aligned}
	\end{equation}
By a standard argument, $\mathbf{u}=\sbr{u_i}_{i\in I}$ is a solution to the subsystem \eqref{subsystem1}. We assert that $\mathbf{u}$ is nontrivial.

If the assertion is false, then we may assume that some components of $\mathbf{u}$ are trivial. Let $\Gamma:=\lbr{i\in I: u_i \equiv 0}$. Then, for each $i\in \Gamma$, we have $u_{i,n}\to 0$ strongly in $L^2(\Omega)$. By Lemma \ref{L6 norm -2} and Sobolev inequality we see that
\begin{equation}\label{Lowerbounded}
\int_\Omega |\nabla u_{i,n}|^2\geq C,
\end{equation}
where $C$ is independent on $n$.
 As $\mathbf{u}_n\in \mathcal{N}_I$ and $\beta_{ij}\leq0$, we get that
\begin{equation*}
	 |\nabla u_{i,n}|_2^2+o(1)=\Ni{u_{i,n}}\leq\beta_{ii}|u_{i,n}|_6^6 \leq \beta_{ii}\widetilde{S}^{-3}|\nabla u_{i,n}|_2^6+o(1).
\end{equation*}
 Combining this with \eqref{Lowerbounded} we know that $\beta_{ii}^{-\frac{1}{2}}\widetilde{S}^{\frac{3}{2}}\leq |\nabla u_{i,n}|_2^2+o(1)$ for every $i\in \Gamma$. Since $\mathbf{u}$ solves \eqref{subsystem1},  we obtain
\begin{equation}
	\begin{aligned}
		\mathcal{C}_I &=\lim_{n \to \infty} J_I(\mathbf{u}_n) =\lim_{n \to \infty} \frac{1}{3}\left(\sum_{i\notin\Gamma}  \Ni{u_{i,n}}+ \sum_{i\in \Gamma}\Ni{u_{i,n}} \right) \\
		& \geq \liminf_{n\to \infty} \frac{1}{3}\sum_{i\notin\Gamma}  \Ni{u_{i,n}} + \frac{1}{3}\sum_{i\in \Gamma}\beta_{ii}^{-\frac{1}{2}}\widetilde{S}^{\frac{3}{2}}\\
		&\geq \frac{1}{3}\sum_{i\notin\Gamma}  \Ni{u_{i}}+\frac{1}{3}\sum_{i\in \Gamma}\beta_{ii}^{-\frac{1}{2}}\widetilde{S}^{\frac{3}{2}}
=J_{I\setminus\Gamma}(\mathbf{u})+\frac{1}{3}\sum_{i\in \Gamma}\beta_{ii}^{-\frac{1}{2}}\widetilde{S}^{\frac{3}{2}}\\
&\geq 	\mathcal{C}_{I\setminus\Gamma} +\frac{1}{3}\sum_{i\in \Gamma}\beta_{ii}^{-\frac{1}{2}}\widetilde{S}^{\frac{3}{2}},
	\end{aligned}
\end{equation}
which leads to a contradiction. Therefore, $\mathbf{u}$ is nontrivial. This implies that $\mathbf{u}\in \mathcal{N}_I$, and
\begin{equation*}
	\mathcal{C}_I \leq J(\mathbf{u}) \leq \liminf_{n\to \infty} J(\mathbf{u}_n) =	\mathcal{C}_I.
\end{equation*}
Hence, $J(\mathbf{u})=	\mathcal{C}_I$.  This completes the proof.
\end{proof}

The following proposition will play an important role in proving that $\mathcal{C}$ is achieved by a solution with $d$ nontrivial components. Our approach is inspired by \cite{ClappSzulkin}.

\begin{proposition} \label{Key estimate}
	Suppose that $\beta_{ij}\leq 0$ for all $i \neq j$, then
	\begin{equation} \label{es-4}
		\mathcal{C}< \min \lbr{\mathcal{C}_I+\frac{1}{3}\sum_{i\notin I}\beta_{ii}^{-\frac{1}{2}}\widetilde{S}^{\frac{3}{2}}:I \subsetneq \lbr{1,\cdots,d}}.
	\end{equation}
\end{proposition}

\begin{proof}
	We proceed to  prove this statement by induction on the number of equations .
	
	For the case $d=1$, this statement was proved by Br\'ezis and Nirenberg in \cite{Brezis-Nirenberg1983}.

Assume that the statement is true for every subsystem with $|I|$$  \leq d-1 $. Then the statement \eqref{es-4} reduces to

\begin{equation} \label{es-5}
	\mathcal{C}< \min \lbr{\mathcal{C}_I+\frac{1}{3}\sum_{i\notin I}\beta_{ii}^{-\frac{1}{2}}\widetilde{S}^{\frac{3}{2}}:|I|=d-1}.
\end{equation}
Without loss of generality, we may assume that $I=\lbr{1,\cdots,d-1}$. By Lemma \ref{achieve2} and our induction hypothesis, there exists a least energy positive solution $\sbr{u_1,\cdots,u_{d-1}}$  to the corresponding subsystem with $I=\lbr{1,\cdots,d-1}$ and $J(u_1,\cdots,u_{d-1})=\mathcal{C}_I$.

For simplicity, we may assume $0\in \Omega$ and $B_{R_0}(0) $ is the largest ball contained in $\Omega$, then we take
\begin{equation}
	\varphi(x)=\begin{cases}
		\cos\sbr{\frac{\pi |x|}{2R_0}},  &x \in B_{R_0}(0),\\
		0 ,& x \in \Omega \setminus B_{R_0}(0),
	\end{cases}
\end{equation}
and  set
\begin{equation*}
	w_{\varepsilon}(x):=U_{\varepsilon}(x)\varphi(x).
\end{equation*}
where $U_{\varepsilon}$ is defined in \eqref{defi of U{varepsilon}}. Similarly to Lemma \ref{GSS energy estimate 1} we get
\begin{equation} \label{3.17}
	\int_{\Omega} |\nabla w_{\varepsilon}|^2 = \widetilde{S}^{\frac{3}{2}}+\frac{\sqrt{3}}{2R_0}\pi^3\varepsilon + O(\varepsilon^2), \quad \int_{\Omega} |w_{\varepsilon}|^6  = \widetilde{S}^{\frac{3}{2}}+ O(\varepsilon^2),
\end{equation}
\begin{equation}\label{3.19}
	\int_{\Omega} |w_{\varepsilon}|^2  = 2\sqrt{3}\pi\varepsilon R_0+ O(\varepsilon^2).
\end{equation}
Moreover,
\begin{equation}\label{ES3-3}
	\int_{\Omega}|w_{\varepsilon}|^3\leq \int_{B_{R_0}(0)}U_\varepsilon^3=C\int_{B_{R_0}(0)}\left(\frac{\varepsilon}{\varepsilon^2+|x|^2}\right)^{\frac{3}{2}}dx
	=C\varepsilon^{\frac{3}{2}}|ln\varepsilon|+O(\varepsilon^2).
\end{equation}
 Note that by the standard regularity theory we have $u_i \in C^{0}(\bar{\Omega})$. Therefore,
\begin{equation} \label{ES3-4}
	\int_{\Omega} |u_i|^3|w_{\varepsilon}|^3   \leq \sbr{\max_{x\in \bar{\Omega}}|u_i(x)|^3}\int_{\Omega} |w_{\varepsilon}|^3 =C\varepsilon^{\frac{3}{2}}|ln\varepsilon|+O(\varepsilon^2).
\end{equation}
To show that  \eqref{es-5} holds, we need the following claim.

{\bf Claim:} There exist $r,R >0 $  independent of $\varepsilon$ and $t_{\varepsilon,1},\cdots, t_{\varepsilon,d}$ $ \in \mbr{r,R}$ such that	$$u_{\varepsilon}=(t_{\varepsilon,1}u_1,\cdots,t_{\varepsilon,d-1}u_{d-1}, t_{\varepsilon,d}w_{\varepsilon}) \in \mathcal{N}.$$

Assume now that this claim is true (we will prove this later). Then
\begin{equation}
	\begin{aligned}
	\mathcal{C}\leq J(u_{\varepsilon})&\leq\frac{1}{2}\sum_{i=1}^{d-1}t_{\varepsilon,i}^2\Ni{u_i}-\frac{1}{6}\sum_{i=1}^{d-1}t_{\varepsilon,i}^6\beta_{ii}|u_i|_6^6-\frac{1}{6}\sum_{\substack{i,j=1,j\neq i}}^{d-1}t_{\varepsilon,i}^3t_{\varepsilon,j}^3\beta_{ij}|u_iu_j|_3^3\\
		&+\frac{1}{2}t_{\varepsilon,d}^2\left\| w_{\varepsilon}\right\|_d^2  -\frac{1}{6}t_{\varepsilon,d}^6\beta_{dd}|w_{\varepsilon}|_6^6+\frac{1}{3}\sum_{i=1}^{d-1}R^4t_{\varepsilon,d}^2|\beta_{id}||u_iw_{\varepsilon}|_3^3\\
		&=: \Psi(t_{\varepsilon,1},\cdots,t_{\varepsilon,d-1})+ \Phi(t_{\varepsilon,d}).
	\end{aligned}
\end{equation}
As $\sbr{u_1,\cdots,u_{d-1}}$ is a least energy positive solution to the corresponding subsystem, then $(1,\cdots,1)$ is a critical point of $\Psi$. By \cite[Lemma 2.2]{ClappSzulkin} we get that the critical point is unique and
\begin{equation*}
 \max_{t_1,...,t_{d-1}>0}\Psi(t_1,...,t_{d-1})=\Psi(1,\cdots,1)=J_I\sbr{u_1,\cdots,u_{d-1}}=\mathcal{C}_I.
\end{equation*}
By \eqref{3.17}, \eqref{3.19}, \eqref{ES3-4} and $R$ is independent of $\varepsilon$, we know that
\begin{equation*}
	\frac{1}{3}\sum_{i=1}^{d-1}R^4t^2|\beta_{id}||u_iw_{\varepsilon}|_3^3= o(\varepsilon)t^2 \quad \text{ for } \varepsilon \text{ small enough },
\end{equation*}
and so
\begin{equation*}
	\Phi(t)=\frac{1}{2}\left(\widetilde{S}^{\frac{3}{2}}+2\sqrt{3}\pi R_{0} (\lambda_d+\frac{\pi^2}{4R_0^2})\varepsilon+o(\varepsilon)+O(\varepsilon^2) \right) t^2-\frac{1}{6}\left(\beta_{dd}\widetilde{S}^{\frac{3}{2}}+O(\varepsilon^2)\right) t^6.
\end{equation*}
Since   $\lambda_d \in (-\lambda_1(\Omega), -\lambda^*(\Omega))$, where $\lambda^*(\Omega)=\frac{\pi^2}{4R_0^2} $, it is standard to see that
\begin{equation*}
	\max_{t>0} \Phi(t) < \frac{1}{3}\beta_{dd}^{-\frac{1}{2}}\widetilde{S}^{\frac{3}{2}} \quad \text{for $\varepsilon$ small enough}.
\end{equation*}
It follows that
\begin{equation}
	\begin{aligned}
		\mathcal{C} \leq \max_{t_1,...,t_{d-1}>0}\Psi(t_1,...,t_{d-1})+ \max_{t>0} \Phi(t)< \mathcal{C}_I+\frac{1}{3}\beta_{dd}^{-\frac{1}{2}}\widetilde{S}^{\frac{3}{2}} \quad \text{for $\varepsilon$ small enough}.
	\end{aligned}
\end{equation}
Therefore, if we assume here that  the claim is true, the proof is completed. It remains to prove this claim.

Consider the polynomial function $\mathcal{J}:(0,+\infty)^d \to \R$
	\begin{equation}
		\begin{aligned}
		\mathcal{J}(\mathbf{t}):= J(t_1u_1,\cdots,t_{d-1}u_{d-1},t_dw_{\varepsilon})=\sum_{i=1}^d a_it_i^2 - \sum_{i=1}^d b_it_i^6 +\sum_{\substack{i,j=1,j\neq i}}^d c_{ij}t_i^3t_j^3,
		\end{aligned}		
	\end{equation}
where $\mathbf{t}=(t_1,\cdots,t_d)$ and
\begin{equation*}
	\begin{aligned}
		&a_i=\frac{1}{2} \Ni{u_i}, \ b_i=\frac{1}{6}\beta_{ii}|u_i|_6^6,\ c_{ij}=-\frac{1}{6}\beta_{ij}|u_iu_j|_3^3,\quad  i,j=1,\ldots,d-1, \\
		& a_d=\frac{1}{2} \|w_{\varepsilon}\|_d^2 , \ b_d=\frac{1}{6}\beta_{dd}|w_{\varepsilon}|_6^6,\ c_{id}=c_{di}=-\frac{1}{6}\beta_{id}|u_iw_{\varepsilon}|_3^3, \quad i=1,\ldots,d-1.
	\end{aligned}
\end{equation*}
Then, for any $i=1,\cdots,d$,
	\begin{equation}
		\partial_i \mathcal{J}(\mathbf{t} )= 2a_it_i-6b_it_i^5+6\sum_{\substack{j=1 ,j\neq i}}^d c_{ij}t_i^2t_i^3.
	\end{equation}
 For  $\varepsilon$ small enough, by  \eqref{ES3-4}  we have
\begin{equation} \label{4.20}
	\begin{aligned}
		6b_i -\sum_{\substack{j=1 ,j\neq i}}^d 6c_{ij}&=\beta_{ii}|u_i|_6^6+\sum_{\substack{j=1 ,j\neq i}}^{d-1 }\beta_{ij}\int_{\Omega} |u_i|^3|u_j|^3+ \beta_{id} \int_{\Omega} |u_i|^3|w_{\varepsilon}|^3\\ &=\Ni{u_i}+ \beta_{id} \int_{\Omega} |u_i|^3|w_{\varepsilon}|^3\geq \frac{1}{2}\Ni{u_i}>0  \quad \text{ for every }1\leq  i\leq d-1,\\
		6b_i -\sum_{\substack{j=1 ,j\neq i}}^d 6c_{ij}&=\beta_{dd}|w_{\varepsilon}|_6^6+\sum_{\substack{j=1}}^{d-1 }\beta_{dj}\int_{\Omega} |w_{\varepsilon}|^3|u_j|^3 > \frac{1}{2}\beta_{dd} \widetilde{S}^{\frac{3}{2}}>0,
	\end{aligned}
\end{equation}
this implies that $$6b_i -\sum_{\substack{j=1 ,j\neq i}}^d 6c_{ij}>0 \text{ for every } 1\leq i \leq d.$$
	For $\varepsilon$ small enough, by \eqref{3.17}-\eqref{ES3-4} we have for $1\leq i \leq d-1$
	\begin{equation*}
		\begin{aligned}
			\cfrac{\left\| w_{\varepsilon}\right\|_d^2 }{\beta_{dd}|w_{\varepsilon}|_6^6}>  \frac{1}{2\beta_{dd}}, \quad \cfrac{\left\| w_{\varepsilon}\right\|_d^2 }{\beta_{dd}|w_{\varepsilon}|_6^6+\sum_{\substack{j=1}}^{d-1 }\beta_{dj} |w_{\varepsilon}u_j|_3^3}<\frac{1}{\beta_{dd}},\quad \cfrac{\Ni{u_i}}{\Ni{u_i}+\beta_{id}|u_iw_{\varepsilon}|_3^3} < 2.
		\end{aligned}
	\end{equation*}
Take
\begin{equation}
	r^4=\min\limits_{1\leq i \leq d-1}\lbr{\cfrac{\Ni{u_i}}{\beta_{ii}|u_i|_6^6},\  \frac{1}{2\beta_{dd}}},\ \ 	R^4=\max\lbr{2,\frac{1}{\beta_{dd}}}.
\end{equation}
 Then we have
\begin{equation}
	2a_it-\sbr{6b_i-6\sum_{\substack{j=1 ,j\neq i}}^d c_{ij}}t^5<0 \quad \text{ if } t\in \sbr{R,\infty},
\end{equation}
and
\begin{equation}
		2a_it-6b_it^5>0\quad \text{ if } t\in \sbr{0,r}.
\end{equation}
Take $\mathbf{s}=(s_1,\cdots,s_d)\in \sbr{0,\infty}^d$, we assume that $s_i=\max\lbr{s_1,\cdots,s_d}$. Then if $s_i>R$, then
\begin{equation}
	\partial_i \mathcal{J}(\mathbf{s} ) \leq  2a_is_i-\sbr{6b_i-6\sum_{\substack{j=1 ,j\neq i}}^d c_{ij}}s_i^5<0.
\end{equation}
On the other hand, if  $s_i<r$, then

\begin{equation}
	\partial_i \mathcal{J}(\mathbf{s} ) \geq 2a_is_i-6b_is_i^5>0.
\end{equation}
This  fact implies that
\begin{equation}
\max\limits_{\mathbf{t}\in \sbr{0,\infty}^d} \mathcal{J}(\mathbf{t} )=\max\limits_{\mathbf{t}\in \mbr{r,R}^d} \mathcal{J}(\mathbf{t} ).
	\end{equation}
In particular, $ \mathcal{J}$ attains its maximum  on $\sbr{0,\infty}^d$, and this maximum point $\sbr{t_{\varepsilon,1},\cdots, t_{\varepsilon,d}} $ must be a critical point. Then it is easy to check	$(t_{\varepsilon,1}u_1,\cdots,t_{\varepsilon,d-1}u_{d-1}, t_{\varepsilon,d}w_{\varepsilon}) \in \mathcal{N}$ and the claim is true. This completes the proof.
\end{proof}

\begin{proof}[\bf Conclusion of the proof of theorem \ref{existence-4-1}]
	Following directly from Lemma \ref{Free critical point}, Lemma \ref{achieve2} and Proposition \ref{Key estimate},  we get that $\mathbf{u}=\sbr{u_1,\cdots,u_d}$ is a nontrivial solution of system \eqref{mainequ} and $J(\mathbf{u})=\mathcal{C}$. Set $\mathbf{\widehat{u}}=(|u_1|,\cdots,|u_d|)$, then $\mathbf{\widehat{u}}$ is  a nonnegative solution of system \eqref{mainequ} and $J(\mathbf{\widehat{u}})=\mathcal{C}$.
By the maximum principle, we see that $\mathbf{\widehat{u}}$ is a least energy positive solution of system \eqref{mainequ}. The proof is completed.
\end{proof}



\begin{thebibliography}{1}

\bibitem{Ambrosetti 2007}A. Ambrosetti, E. Colorado, {\it Standing waves of some coupled nonlinear Schr\"{o}dinger equations}. J. Lond. Math. Soc.  75 (2007), 67--82.

\bibitem{Atkinson-Brezis-Peletier} F.V. Atkinson, H. Brezis, L.A. Peletier, {\it Nodal solutions of elliptic equations with critical Sobolev exponents}. J. Differ. Equ. 85(1), 151--C170 (1990)

\bibitem{Aubin=JDG=1976} T. Aubin, {\it Probl\`emes isop$\acute{e}$rim$\acute{e}$triques et espaces de Sobolev. }Journal of Differential Geometry. 11 (1976), 573--598.

 \bibitem{Bartsch-Dancer-Wang 2010} T. Bartsch, N. Dancer, Z. Q. Wang, {\it A Liouville theorem, a priori bounds, and bifurcating branches of positive solutions for a nonlinear elliptic system.} Calc. Var. Partial Differential Equations. 37 (2010), 345-361.

\bibitem{Brezis-Nirenberg1983} H. Br\'ezis, L. Nirenberg,{\it Positive solutions of nonlinear elliptic equations involving critical Sobolev exponents.}  Comm. Pure Appl. Math.  36 (1983), 437--477.

\bibitem{Brezis Lieb lemma} H. Br\'{e}zis,  E. H. Lieb, {\it A relation between pointwise convergence of functions and convergence of functionals}. Proc. Amer. Math. Soc. 88 (1983), 486--490.

\bibitem{BSWang 2016} J. Byeon, Y. Sato, Z. Q. Wang, {\it Pattern formation via mixed attractive and repulsive interactions for nonlinear Schr\"{o}dinger systems}. J. Math. Pures Appl. 106 (2016), 477--511.

\bibitem{BLWang 2019} J. Byeon, Y. Lee, Z. Q. Wang, {\it Formation of radial patterns via mixed attractive and repulsive interactions for Schr\"{o}dinger systems}. SIAM J. Math. Anal. 51 (2019), 1514--1542.

\bibitem{Cerami-Solimini-Struwe 1986} G. Cerami, S. Solimini and M. Struwe, {\it Some existence results for superlinear elliptic boundary value problems involving critical exponents}. J. Funct. Anal. 69 (1986), 289--306.


\bibitem{Zou 2012}Z. J. Chen, W. M. Zou, {\it Positive least energy solutions and phase separation for coupled Schr\"{o}dinger equations with critical exponent.} Arch. Ration. Mech. Anal. 205 (2012), 515--551.

\bibitem{Chen-Zou2012} Z. J. Chen, N. Shioji, W. M. Zou, {\it Ground state and multiple solutions for a critical exponent problem}. Nonlinear Differential Equations Appl. 19 (2012), 253--277.

\bibitem{Chen-Zou2013}  Z. J. Chen, W. M. Zou, {\it An optimal constant for the existence of least energy solutions of a coupled Schr\"{o}dinger system}. Calc. Var. Partial Differential Equations 48 (2013), 695--711.
	
					
\bibitem{Zou 2015} Z. J. Chen, W. M. Zou, {\it Positive least energy solutions and phase separation for coupled Schr\"{o}dinger equations with critical exponent: higher dimensional case.} Calc. Var. Partial Differential Equations. 52 (2015), 423--467.


\bibitem{Clapp-Weth 2005} M. Clapp, T. Weth, {\it Multiple solutions for the Brezis-Nirenberg problem}. Adv.
Differential Equations. 10 (2005) 463--480.

\bibitem{Clapp-Pistoia2018} M. Clapp, A. Pistoia, {\it Existence and phase separation of entire solutions to a pure critical competitive elliptic system}. Calc. Var. Partial Differential Equations. 57 (2018).

\bibitem{ClappSzulkin} M. Clapp, A. Szulkin, {\it A simple variational approach to weakly coupled competitive elliptic systems.}  NoDEA 26, (2019).

\bibitem{Correia/Oliveria/Tavares=JFA=2016}  S.  Correia, F. Oliveira, H. Tavares, {\it Semitrivial vs. fully nontrivial ground states in cooperative cubic Schr\"{o}dinger systems with d  $ \geq $ 3 equations.} J. Funct. Anal. 271 (2016), 2247--2273.	

\bibitem{GuoZou2018}Y.-X. Guo, S. P. Luo, W. M. Zou, {\it The existence, uniqueness and nonexistence of the ground state to the N-coupled Schr\"{o}dinger systems in $\mathbb{R}^{n} (N\leq 4)$}. Nonlinearity. 31 (2018), 314--339.


\bibitem{HeYang2018} Q.-H. He, J. Yang, { \it Quantitative properties of ground-states to an M-coupled system with critical exponent in $\mathbb{R}^{N}$}. Sci. China Math. 61 (2018), 709--726.

\bibitem{Kim2013}S. Kim, {\it On vector solutions for coupled nonlinear Schr\"{o}dinger equations with critical exponents}. Commun. Pure Appl. Anal. 12 (2013) 1259--1277.

\bibitem{Lin-Wei=CMP=2005} T. C. Lin, J. C. Wei, {\it Ground State of $N$ Coupled Nonlinear Schr\"{o}dinger Equations in $\R^n$,$n\leq$3}. Commun. Math. Phys. 255 (2005), 629--653.

\bibitem{Mandel=NoDEA=2015}  R. Mandel, {\it Minimal energy solutions for cooperative nonlinear Schr\"{o}dinger systems.} Nonlinear Differ. Equ. Appl. 22 (2015), 239--262.

\bibitem{Pistoia 2018-1} A. Pistoia, N. Soave, {\it On Coron's problem for weakly coupled elliptic systems}. Proc. Lond. Math. Soc. 116 (1) (2018), 33--67.

\bibitem{Tavares 2019} A. Pistoia, N. Soave, H. Tavares, {\it A fountain of positive Bubbles on a Coron$'$s Problem for a Competitive Weakly Coupled Gradient System}.  J. Math. Pures Appl. 135 (9) (2020), 159--198.

\bibitem{Roselli-Willem 2009} P. Roselli, M. Willem, {\it Least energy nodal solutions of the Brezis-Nirenberg problem in dimension $N = 5$}. Comm. Contemp. Math. 11 (2009), 59--69.



\bibitem{Sato-Wang 2015} Y. Sato, Z.-Q. Wang, {\it Least energy solutions for nonlinear Schr\"{o}dinger systems with mixed attractive and repulsive couplings}. Adv. Nonlinear Stud. 15 (2015), 1--22.

\bibitem{Schechter-Zou} M.Schechter, W.M. Zou, {\it On the Brezis-Nirenberg problem}. Arch. Ration. Mech. Anal. 197 (2010), 337--356.
					
\bibitem{soave-hugo=JDE=2016} N. Soave, H. Tavares, {\it New existence and symmetry results for least energy positive solutions of Schr\"{o}dinger systems with mixed competition and cooperation terms}.
					 Journal of Differential Equations. 261 (2016), 505--537.

\bibitem{Soave 2015} N. Soave, {\it On existence and phase separation of solitary waves for nonlinear Schr\"{o}dinger systems modelling simultaneous cooperation and competition}. Calc. Var. Partial Differential Equations. 53 (3) (2015), 689--718.

\bibitem{Sirakov 2007} B. Sirakov, {\it Least energy solitary waves for a system of nonlinear Schr\"{o}dinger equations in $\mathbb{R}^{n}$}. Comm. Math. Phys. 271 (1) (2007), 199--221.
					
\bibitem{TavaresYou2020} H. Tavares, S. You, {\it Existence of least energy positive solutions to Schr\"{o}dinger systems with mixed competition and cooperation terms: the critical case}. Calc. Var. Partial Differential Equations. 59, (2020).

\bibitem{Hugo-You-Zou=Arxiv} H. Tavares, S. You, W. M. Zou, {\it Least energy positive solutions of critical Schr{\"o}dinger systems with mixed competition and cooperation terms: the higher dimensional case.} arXiv:2109.14753

\bibitem{Talenti=AMPA=1976} G. Talenti. { \it Best constant in Sobolev inequality.} Ann. Mat. Pura Appl. 110 (1976), 353--372.

\bibitem{Timmermans 1998}E. Timmermans, {\it Phase separation of Bose-Einstein condensates}. Phys. Rev. Lett. 81 (26) (1998), 5718--5721.

\bibitem{WT 2008} J. C. Wei, T. Weth, {\it Radial solutions and phase separation in a system of two coupled Schr\"{o}dinger equations}. Arch. Ration. Mech. Anal. 190 (2008), 83--106.

					
\bibitem{william=1996} M. Willem. {\it Minimax Theorems.} Birkh\"{a}user Boston.1996.
					

\bibitem{Wu Y-Z 2017} Y. Z. Wu,  {\it On a $K$-component elliptic system with the Sobolev critical exponent in high dimensions: the repulsive case}. Calc. Var. Partial Differential Equations. 56, 2017.

\bibitem{YePeng2014} H. Y. Ye, Y. F. Peng, {\it Positive least energy solutions for a coupled Schr\"{o}dinger system with critical exponent}. J. Math. Anal. Appl. 417 (2014), 308--326.
					
\bibitem{Yin-Zou} X. Yin, W. M. Zou, {\it Positive least energy solutions for $k$-coupled Schr\"{o}dinger system with critical exponent: the higher dimension and cooperative case}. J. Fixed Point Theory Appl. 24 (2022).

\bibitem{You-Zou2022} S. You, W. M. Zou, {\it Existence of least energy positive solutions to critical Schr\"{o}dinger systems in $\R^3$}. Appl. Math. Lett. 128 (2022).				
				

				







\end{thebibliography}
\end{document}